\documentclass[a4paper,english]{jnsao}
\usepackage[utf8]{inputenc}
\usepackage[T1]{fontenc}
\usepackage{csquotes}
\usepackage{babel}
\usepackage{pgfplots}
\usepgfplotslibrary{colorbrewer}
\usepackage{float}
\usepackage{enumitem}
\usepackage{comment}
\usepackage{mathtools}
\usepackage{algpseudocode}
\usepackage[nameinlink,capitalize]{cleveref}
\usepackage[textsize=footnotesize,color=DarkOrange!40]{todonotes}
\newcommand{\term}{\emph}

\newcommand{\field}[1]{\mathbb{#1}}
\newcommand{\N}{\mathbb{N}}

\newcommand{\R}{\field{R}}
\newcommand{\extR}{\overline \R}

\newcommand{\norm}[1]{\|#1\|}

\newcommand{\abs}[1]{|#1|}

\newcommand{\inv}[1]{#1^{-1}}
\newcommand{\grad}{\nabla}
\newcommand{\freevar}{\,\boldsymbol\cdot\,}

\newcommand{\Union}\bigcup
\newcommand{\Isect}\bigcap
\newcommand{\union}\cup
\newcommand{\isect}\cap
\newcommand{\bigunion}\bigcup
\newcommand{\bigisect}\bigcap

\newcommand{\defeq}{:=}

\newcommand{\downto}{\searrow}
\newcommand{\upto}{\nearrow}
\newcommand{\subdiff}{\partial}

\DeclareMathOperator*{\argmin}{arg\,min}

\DeclareMathOperator{\interior}{int}

\DeclareMathOperator{\Dom}{dom}

\DeclareMathOperator{\epi}{epi}

\DeclareMathOperator{\diag}{diag}

\DeclareMathOperator{\Sym}{Sym}
\DeclareMathOperator{\trace}{trace}

\def \uminusSym{\setbox0=\hbox{$\cup$}\rlap{\hbox 
        to\wd0{\hss\raise0.5ex\hbox{$\scriptscriptstyle{-}$}\hss}}\box0}
    
\newcommand{\iprod}[2]{\langle #1,#2\rangle}

\def\llangle{\langle\kern-3pt\langle}
\def\rrangle{\rangle\kern-3pt\rangle}

\def \weaktostarSym{\setbox0=\hbox{$\rightharpoonup$}\rlap{\hbox 
        to\wd0{\hss\raise1ex\hbox{$\scriptscriptstyle{*\,}$}\hss}}\box0}
    
\def\linear{\mathbb{L}}

\newcommand{\setto}{\rightrightarrows}

\def\extR{\overline \R}

\def\opt#1{\bar #1}

\def\this#1{#1^k}
\def\nexxt#1{#1^{k+1}}

\def\optu{{\opt{u}}}
\def\optx{{\opt{x}}}

\def\opty{{\opt{y}}}

\def\nextu{\nexxt{u}}
\def\nextx{\nexxt{x}}

\def\thisu{\this{u}}
\def\thisx{\this{x}}

\def\thisv{\this{v}}

\def\invstar#1{{#1}^{-1,*}}

\newcommand{\Step}{T}

\def\thisw{\this{w}}
\def\nextw{\nexxt{w}}

\DeclareMathOperator{\prox}{prox}

\def\d{\,d}

\def\dualprod#1#2{\langle #1|#2\rangle}

\let\phi=\varphi
\let\epsilon=\varepsilon

\def\Id{\operatorname{Id}}
\def\Inj{\operatorname{In}}

\def\wop{B}
\def\linwop{B_x}
\def\constwop{B_{\mathop{\mathrm{const}}}}
\def\xop{\bar\grad_x B}
\def\uop{B_u}
\def\xuop{B_{xu}}
\def\wrhs{L}
\def\SupNorm{\mathscr{S}}
\def\RealXStep{\delta}

\def\outerprod{:}

\renewrobustcmd{\downto}{{{\mathchoice%
            {\rotatebox[origin=c]{-20}{$\to$}}
            {\rotatebox[origin=c]{-20}{$\to$}}
            {\rotatebox[origin=c]{-20}{\scalebox{0.75}{$\to$}}}
            {\rotatebox[origin=c]{-20}{\scalebox{0.6}{$\to$}}}
}}}

\renewrobustcmd{\upto}{{{\mathchoice%
            {\rotatebox[origin=c]{20}{$\to$}}
            {\rotatebox[origin=c]{20}{$\to$}}
            {\rotatebox[origin=c]{20}{\scalebox{0.75}{$\to$}}}
            {\rotatebox[origin=c]{20}{\scalebox{0.6}{$\to$}}}
}}}

\def\iprodx#1{\left\langle #1\right\rangle}

\def\optw{\opt w}
\def\nxt#1{{#1}^{k+1}}
\def\nxxt#1{\bar{#1}^{k+1}}

\pgfplotsset{
    compat=1.3,
    every axis label/.append style = {font = \scriptsize},
    every tick label/.append style = {font = \scriptsize},
    xlabel near ticks,
    ylabel near ticks,
    halfpage style/.style={
        width=.475\linewidth,
        height=.475\linewidth,
    },
    /pgfplots/ylabel atop ticks/.style={
        /pgfplots/every axis y label/.style={
            at={(ticklabel cs:1)},
            anchor=south west,
        },
        ylabel shift = -1ex,
    },
    convergence style/.style={
        scale only axis,
        ylabel atop ticks,
        xticklabel style={overlay},
        trim axis right,
        trim axis left,
        outer sep = 0pt,
        scaled x ticks=false,
        xminorticks=true,
        minor x tick num=1,
        yminorticks=true,
        minor y tick num=3,
        grid=both,
        major grid style={dotted,Greys-H},
        minor grid style={dotted,Greys-E},
        axis x line*=bottom,
        axis y line*=left,
        outer sep=0pt,
        every axis plot post/.style={
            line width=0.7pt
        },
        cycle list/Dark2,
        cycle multiindex* list={
            mark list*\nextlist
            Dark2\nextlist
        },
        mark repeat={20},
        mark size = 1pt,
        line width = 0.7pt,
    },
    legend ne style/.style={
        legend style={
            legend pos=north east,
            inner sep=0pt,
            outer sep=0pt,
            legend cell align=left,
            align=left,
            draw=none,
            fill=none,
            font=\scriptsize
        },
    },
    legend se style/.style={
        legend style={
            legend pos=south east,
            inner sep=0pt,
            outer sep=0pt,
            legend cell align=left,
            align=left,
            draw=none,
            fill=none,
            font=\scriptsize
        },
    },
    parameter style/.style={
        trim axis right,
        trim axis left,
        enlargelimits=false,
        axis on top,
        axis equal image,
        scale only axis,
        height=0.27\linewidth
    },
    colorbar style/.style={
        trim axis left,
        trim axis right,
        enlargelimits=false,
        axis on top,
        axis equal image,
        xtick style={draw=none},
        ytick style={draw=none},
        xmajorticks=false,
        ylabel near ticks,
        yticklabel pos=right,
        scale only axis,
        height=0.27\linewidth,
        anchor = west,
        xshift = 2ex,
    }
}
\newcommand{\findminmax}[4]{
    \pgfplotstableread{#1}{#4}
    \def\tempa{#3}
    \def\tempb{true}
    \ifx\tempa\tempb
        \pgfplotstablegetelem{0}{#2}\of{#4}%
        \pgfmathparse{\pgfplotsretval}
        \global\let\ExperimentMaxValue\pgfmathresult
        \pgfplotstablegetelem{0}{#2}\of{#4}%
        \pgfmathparse{\pgfplotsretval}
        \global\let\ExperimentMinValue\pgfmathresult
    \fi
    \pgfplotstablegetrowsof#4
    \pgfmathparse{\pgfplotsretval-1}
    \foreach \i in {0,...,\pgfmathresult}{
        \pgfplotstablegetelem{\i}{#2}\of{#4}%
        \pgfmathparse{max(\ExperimentMaxValue, \pgfplotsretval)}
        \global\let\ExperimentMaxValue\pgfmathresult
        \pgfplotstablegetelem{\i}{#2}\of{#4}%
        \pgfmathparse{min(\ExperimentMinValue, \pgfplotsretval)}
        \global\let\ExperimentMinValue\pgfmathresult
    }
}

\newcommand{\SetMinMax}[3]{%
    \findminmax{#1}{#2}{true}{#3}
}

\newcommand{\UpdMinMax}[3]{
    \findminmax{#1}{#2}{false}{#3}
}

\pgfplotsset{
    fixedylog/.code 2 args={
        \pgfmathsetmacro{\logscale}{#1}
        \pgfmathsetmacro{\nonlogmax}{1+\logscale}
        \pgfmathsetmacro{\logmax}{log10(\nonlogmax)}
        \gdef\scaledlog##1{log10(1+\logscale*((##1-\ExperimentMinValue)/(\ExperimentMaxValue-\ExperimentMinValue)))}
        \gdef\delogfracvalue##1{
            \pgfkeys{/pgf/number format/.cd,#2}
            \pgfmathparse{((10^(##1*\logmax)-1)/\logscale)*(\ExperimentMaxValue-\ExperimentMinValue)+\ExperimentMinValue}
            \pgfmathprintnumber{\pgfmathresult}
        }
        \pgfplotstableset{
            y filter/.expression={\scaledlog{y}},
            ytick={0, 0.25*\logmax, 0.50*\logmax, 0.75*\logmax, \logmax},
            yticklabels={
                \delogfracvalue{0},
                \delogfracvalue{0.25},
                \delogfracvalue{0.50},
                \delogfracvalue{0.75},
                \delogfracvalue{1.0},
            }
        }
    }
}

\manuscriptcopyright{}
\manuscriptlicense{}

\theoremstyle{definition}
\newtheorem{assumption}[theorem]{Assumption}
\crefname{assumption}{Assumption}{Assumptions}
\theoremstyle{definition}
\newtheorem{experiment}{Experiment}
\crefname{experiment}{Experiment}{Experiments}

\floatstyle{ruled}
\floatname{algorithm}{Algorithm}
\newfloat{algorithm}{tbp!}{loa}
\numberwithin{algorithm}{section}

\title{A nonsmooth primal-dual method with interwoven PDE constraint solver}
\shorttitle{Primal-dual method with interwoven PDE constraint solver}

\date{2022-11-09 (revised 2024-05-23)}
\author{
    Bjørn Jensen\thanks{%
        Department of Mathematical Information Technology, University of Jyväskylä, Finland
        \email{bjorn.c.s.jensen@jyu.fi}
    }
    \and
    Tuomo Valkonen\thanks{%
        ModeMat, Escuela Politécnica Nacional, Quito, Ecuador
        \emph{and}
        Department of Mathematics and Statistics, University of Helsinki, Finland
        \email{tuomo.valkonen@iki.fi}
    }
}

\acknowledgements{%
    This research was supported by the Academy of Finland decisions 314701, 320022, and 345389.
}

\begin{document}

\maketitle

\begin{abstract}
    We introduce an efficient first-order primal-dual method for the solution of nonsmooth PDE-constrained optimization problems.
    We achieve this efficiency through \emph{not} solving the PDE or its linearisation on each iteration of the optimization method.
    Instead, we run the method interwoven with a simple conventional linear system solver (Jacobi, Gauss–Seidel, conjugate gradients), always taking only \emph{one step} of the linear system solver for each step of the optimization method. The control parameter is updated on each iteration as determined by the optimization method.
    We prove linear convergence under a second-order growth condition, and numerically demonstrate the performance on a variety of PDEs related to inverse problems involving boundary measurements.
\end{abstract}

\section{Introduction}
\label{sec:introduction}

Our objective is to develop efficient first-order algorithms for the solution of PDE-constrained optimization problems of the type
\[
    \min_{x,u} F(x) + Q(u) + G(Kx)
    \quad\text{subject to}\quad
    \wop (u, w; x)=\wrhs w \quad\text{for all}\quad w,
\]
where $K$ is a linear operator and the functions $F$, $G$, and $Q$ are convex but the first two possibly nonsmooth. The functionals $B$ and $L$ model a partial differential equation in weak form, parametrised by $x$; for example, $B(u, w; x)=\iprod{\grad u}{x\grad w}$.

Semismooth Newton methods \cite{mifflin1977semismooth,qi1993nonsmooth} are conventionally used for such problems when a suitable reformulation exists \cite{hintermuller2002primaldual,kunisch2008lagrange,ulbrich2002semismooth,ulbrich2011semismooth,hintermuller2006infeasible}.
Reformulations may not always be available, or yield effective algorithms.
The solution of large linear systems may also pose scaling challenges.
Therefore, first-order methods for PDE-constrained optimization have been proposed \cite{tuomov-pdex2nlpdhgm,tuomov-nlpdhgm-redo,tuomov-nlpdhgm-block,tuomov-nlpdhgm-general} based on the primal-dual proximal splitting (PDPS) of \cite{chambolle2011first}. The original version applies to convex problems of the form
\begin{equation}
    \label{eq:intro:prob-pdps}
    \min_x F(x) + G(Kx).
\end{equation}
The primal-dual expansion permits efficient treatment of $G \circ K$ for nonsmooth $G$. In \cite{tuomov-pdex2nlpdhgm,tuomov-nlpdhgm-redo,tuomov-nlpdhgm-block,tuomov-nlpdhgm-general} $K$ may be nonlinear, such as the solution operator of a nonlinear PDE.

However, first-order methods generally require a very large number of iterations to exhibit convergence. If the iterations are cheap, they can, nevertheless, achieve good performance. If the iterations are expensive, such as when a PDE needs to be solved on each step, their performance can be poor.
Therefore, especially in inverse problems research, Gauss–Newton -type approaches are common for \eqref{eq:intro:prob-pdps} with nonlinear $K$; see, e.g., \cite{darde2021electrodeless,vilhunen2002simultaneous,jauhiainen2019gaussnewton}. They are easy: first linearise $K$, then apply a convex optimization method or, in simplest cases, a linear system solver. Repeat. Even when a first-order method is used for the subproblem, Gauss–Newton methods can be significantly faster than full first-order methods \cite{jauhiainen2019gaussnewton} if they converge at all \cite{tuomov-nlpdhgm}.
This stems from the following and only practical difference between the PDPS for nonlinear $K$ and Gauss–Newton applied to \eqref{eq:intro:prob-pdps} with PDPS for the inner problems: the former re-linearizes and factors $K$ on \emph{each} PDPS iteration, the latter only on each outer Gauss–Newton iteration.

In this work, we avoid forming and factorizing the PDE solution operators altogether by \emph{running an iterative solver for the constantly adapting PDE simultaneously with the optimization method}.
This may be compared to the approach to bilevel optimization in \cite{suonpera2022bilevel}.
We concentrate on the simple Jacobi and Gauss–Seidel splitting methods for the PDE, while the optimization method is based on the PDPS, as we describe in \cref{sec:algorithm}.
We prove convergence in \cref{sec:convergence} using the testing approach introduced in \cite{tuomov-proxtest} and further elucidated in \cite{clasonvalkonen2020nonsmooth}.
We explain how standard splittings and PDEs fit into the framework in \cref{sec:examples}, and finish with numerical experiments in \cref{sec:numerics}.

Pseudo-time-stepping one-shot methods have been introduced in \cite{ta1991one} and further studied, among others, in \cite{sirignano2022online,kaland2013oneshot,hamdi2010properties,hamdi2009reduced,guenther2016simultaneous,bosse2014oneshot,griewank2006projected,hazra2004simultaneous}.
A “one-shot” approach, as opposed to an “all-at-once” approach, solves the PDE constraints on each step, instead of considering them part of a unified system of optimality conditions.
The aforementioned works solve these constraints inexactly through “pseudo-”time-stepping. This corresponds to the trivial split $A_x=(A_x-\Id)+\Id$ where $A_x$ is such that $\iprod{A_xu}{w}=\wop(u, w; x)$. We will, instead, apply Jacobi, Gauss–Seidel or even \mbox{(quasi-)}conjugate gradient splitting on $A_x$.
In \cite{griewank2006projected,bosse2014oneshot} Jacobi and Gauss–Seidel updates are used for the control variable, but not for the PDEs.
The authors of \cite{hazra2004simultaneous} come closest to introducing non-trivial splitting of the PDEs via Hessian approximation.
However, they and the other aforementioned works generally restrict themselves to smooth problems and employ gradient descent, Newton-type methods, or sequential quadratic programming (SQP) for the control variable $x$.
Our focus is on nonsmooth problems involving, in particular, total variation regularization $G(Kx)=\norm{\grad x}_1$.

\subsection*{Notation and basic results}

Let $X$ be a normed space. We write $\dualprod{\freevar}{\freevar}$ for the dual product and, in a Hilbert space, $\iprod{\freevar}{\freevar}$ for the inner product.
The order of the arguments in the dual product is not important when the action is obvious from context.
For $X$ a Hilbert space, we denote by $\Inj_X : X \hookrightarrow X^*$ the canonical injection, $\dualprod{\Inj_X x}{\tilde x}=\iprod{x}{\tilde x}$ for all $x, \tilde x \in X$.

We write $\linear(X; Y)$ for the space of bounded linear operators between $X$ and $Y$.
We write $\Id_X = \Id \in \linear(X; X)$ for the identity operator on $X$.
If $M \in \linear(X; X^*)$ is non-negative and self-adjoint, i.e., $\dualprod{Mx}{y}=\dualprod{x}{My}$ and $\dualprod{x}{Mx} \ge 0$ for all $x, y \in X$, we define $\norm{x}_M \defeq \sqrt{\dualprod{x}{Mx}}$. Then the \emph{three-point identity} holds:
\begin{equation}
    \label{eq:pythagoras}
    \dualprod{M(x-y)}{x-z} = \frac{1}{2}\norm{x-y}^2_M - \frac{1}{2}\norm{y-z}^2_M + \frac{1}{2}\norm{x-z}^2_M \qquad \text{for all } x,y,z\in X.
\end{equation}
We extensively use the vector Young's inequality
\begin{equation}
    \label{eq:prox}
    \dualprod{x}{y} \le \frac{1}{2a}\norm{x}_X^2 + \frac{a}{2}\norm{y}_{X^\ast}^2
    \quad (x \in X,\,y \in X^\ast,\, a>0).
\end{equation}
These expressions hold in Hilbert spaces also with the inner product in place of the dual product. We write $M^{\star}$ for the inner product adjoint of $M$, and $M^*$ for the dual product adjoint.

We write $\Dom F$ for the effective domain, and $F^*$ for the Fenchel conjugate of $F: X \to \extR \defeq [-\infty, \infty]$.
We write $F'(x) \in X^*$ for the Fréchet derivative at $x$ when it exists, and, if $X$ is a Hilbert space, $\grad F(x) \in X$ for its Riesz presentation.
For convex $F$ on a Hilbert space $X$, we write $\subdiff F(x) \subset X$ for the subdifferential at $x \in X$ (or, more precisely, the corresponding set of Riesz representations, but aside from a single proof in \cref{sec:oc}, we will not be needing subderivatives as elements of $X^*$).
We then define the proximal map
\[
    \prox_{F}(x) \defeq (\Id + \subdiff F)^{-1}(x) = \argmin_{\tilde x\in X}\left\{F(\tilde x) + \frac{1}{2}\norm{\tilde x - x}_{X}^2\right\},\quad x \in X.
\]
We denote the $\{0,\infty\}$-valued indicator function of a set $A$ by $\delta_A$.

We occasionally apply operations on $x \in X$ to all elements of sets $A \subset X$, writing $\dualprod{x+A}{z} \defeq \{\dualprod{x+a}{z} \mid a \in A \}$. For $B \subset \R$, we write $B \ge c$ if $b \ge c$ for all $b \in B$.

On a Lipschitz domain $\Omega \subset \R^n$, we write $\trace_{\partial \Omega} \in \linear(H^1(\Omega); L^2(\partial \Omega))$ for the trace operator on the boundary $\partial \Omega$.

\section{Problem and proposed algorithm}
\label{sec:algorithm}

We start by introducing in detail the type of problem we are trying to solve.
We then rewrite in \cref{sec:algorithm:problem} its optimality conditions in a form suitable for developing our proposed method in \cref{sec:algorithm:algorithm}.
Before this we recall the structure and derivation of the basic PDPS in \cref{sec:algorithm:pdps}.

\subsection{Problem description}
\label{sec:algorithm:problem}

Our objective is to solve
\begin{equation}
    \label{eq:algorithm:problem}
    \min_x J(x) \defeq F(x) + Q(S(x)) + G(K x),
\end{equation}
where $F: X \to \extR$, $G: Y \to \extR$, and $Q: U \to \R$ are convex, proper, and lower semicontinuous on Hilbert spaces $X$, $U$, and $Y$ with $Q$ Fréchet differentiable.
We assume $K \in \linear(X; Y)$ while $S: X \ni x \mapsto u \in U$ is a solution operator of the weak PDE
\begin{equation}
    \label{eq:pde-u}
    \wop(u, w; x) = \wrhs w \quad\text{for all}\quad w \in W.
\end{equation}
Here $L \in U^\ast $ and $B: U \times W \times X \to \R$ is continuous, and affine-linear-affine in its three arguments.
The space $W$ is Hilbert, possibly distinct from $U$ to model nonhomogeneous boundary conditions.
For this initial development, we will tacitly assume unique $S(x)$ and $\grad S(x)$ to exist for all $x \in \Dom F$, but later on in the manuscript, do not directly impose this restriction, or use $S$.

\begin{example}[A linear PDE]
    \label{ex:algorithm:b-idea-linear}%
    On a Lipschitz domain $\Omega \subset \R^n$, consider the PDE
    \[
        \left\{
            \begin{array}{ll}
                \grad\cdot\grad u = x, & \text{on } \Omega,
                \\
                u = g, & \text{on } \partial \Omega.
            \end{array}
        \right.
    \]
    For the weak form \eqref{eq:pde-u} we can take the spaces $U = H^1(\Omega)$,
    $W = H_0^1(\Omega) \times H^{1/2}(\partial\Omega)$, and $X = L^2(\Omega)$.
    Writing $w = (w_\Omega, w_\partial)$, we then set
    \[
        B(u, w; x)
        =
        \iprod{\grad u}{\grad w_\Omega}_{L^2(\Omega)}
        - \iprod{x}{w_\Omega}_{L^2(\Omega)}
        + \iprod{\trace_{\partial\Omega} u}{w_\partial}_{L^2(\partial \Omega)}
        \quad\text{and}\quad
        Lw \defeq \iprod{g}{w_\partial}_{L^2(\partial \Omega)}.
    \]
\end{example}

\begin{example}[A nonlinear PDE]
    \label{ex:algorithm:b-idea}
    On a Lipschitz domain $\Omega \subset \R^n$, consider the PDE
    \[
        \left\{
            \begin{array}{ll}
                \grad\cdot(x\grad u) = 0, & \text{on } \Omega,
                \\
                u = g, & \text{on } \partial \Omega.
            \end{array}
        \right.
    \]
    For the weak form \eqref{eq:pde-u} we can take the spaces $U \subset H^1(\Omega)$,
    $W \subset H_0^1(\Omega) \times H^{1/2}(\partial\Omega)$, and $X \subset L^2(\Omega)$, such that at least one of these subspaces ensures the corresponding $x$, $\grad u$, or $\grad w$ to be in the relevant $L^\infty$ space. This, in practise, requires one of the subspaces to be finite-dimensional, or $X$ to be $H^k(\Omega)$ for $k > n/2$, such that the boundedness of $\Omega$ and Sobolev's inequalities provide the $L^\infty$ bound.
    The latter is an option in infinite-dimensional theory, but in finite-dimensional realisations, it is desirable to use a standard $2$-norm in $X$, as proximal operators and gradient steps with respect to $H^k$-norms (for $k > 0$) are computationally expensive.
    Writing $w = (w_\Omega, w_\partial)$, we then set
    \[
        B(u, w; x)
        =
        \iprod{x\grad u}{\grad w_\Omega}_{L^2(\Omega)}
        + \iprod{\trace_{\partial\Omega} u}{w_\partial}_{L^2(\partial \Omega)}
        \quad\text{and}\quad
        Lw \defeq \iprod{g}{w_\partial}_{L^2(\partial \Omega)}.
    \]
    To ensure the coercivity of $B(\freevar, \freevar; x)$, and hence the existence of unique solutions to \eqref{eq:pde-u}, we will further need to restrict $x$ through $\Dom F$.
\end{example}

We require the sum and chain rules for convex subdifferentials to hold on $F + G \circ K$.
This is the case when
\begin{equation}
    \label{eq:algorithm:sumrule-ok}
    \text{there exists an } x\in \Dom (G\circ K) \isect \Dom F \text{ with } Kx\in \interior(\Dom G).
\end{equation}
We refer to \cite{clasonvalkonen2020nonsmooth} for basic results and concepts of infinite-dimensional convex analysis.
Then by the Fréchet differentiability of $Q$ and the compatibility of limiting (Mordukhovich) subdifferentials (denoted $\subdiff_{M}$) with Fréchet derivatives and convex subdifferentials \cite{mordukhovich2006variational,clasonvalkonen2020nonsmooth},
\[
    \subdiff_{M} J(x) = \subdiff F(x) + \grad S(x)^{\star} \grad Q(S(x)) + K^{\star} \subdiff G(Kx).
\]
Therefore, the Fermat principle for limiting subdifferentials and simple rearrangements (see \cite{tuomov-nlpdhgm,tuomov-nlpdhgm-redo} or \cite[Chapter 15]{clasonvalkonen2020nonsmooth}) establish for \eqref{eq:algorithm:problem} in terms of $(\opt u, \optw, \opt x, \opt y) \in U \times W \times X \times Y$ the necessary first-order optimality condition
\begin{equation}
    \label{eq:algorithm:oc0}
    \left\{\begin{aligned}
        \optu & = S(\optx), \\
        - \grad S(\opt x)^{\star} \grad Q(\optu) -  K^{\star} \opt y & \in \partial F(\opt x), \\
         K \opt x & \in \partial G^*(\opt y).
    \end{aligned}\right.
\end{equation}
We recall that $G^*: Y \to \extR$ is the Fenchel conjugate of $G$.

The term $\grad S(\opt x)^{\star}\grad Q(\optu)$ involves the solution $\opt u$ to the original PDE and the solution $\optw$ to an adjoint PDE. We derive it from a primal-dual reformulation of \eqref{eq:algorithm:problem}.
To do this, we first observe that since $B$ is affine in $x$, it can be decomposed as
\begin{equation}
    \label{eq:algorithm:b-decomposition}
    B(u, w; x) = \linwop(u, w; x) + \constwop(u, w),
\end{equation}
where, $\linwop: U \times W \times X \to \R$ is affine-linear-linear, and $\constwop: U \times W \to \R$ is affine-linear. Indeed $\constwop(u, w) = B(u, w; 0)$, and $\linwop(u, w; x) = \wop(u, w; x) - \wop(u, w; 0)$.
We then introduce the Riesz representation $\xop(u, w)$ of $\linwop(u, w; \freevar) \in X^*$.
Thus
\begin{equation}
    \label{eq:algorithm:riesz}
    \iprod{\xop(u, w)}{x}_X = \linwop(u, w; x)
    \quad\text{for all } u \in U,\, w \in W,\, x \in X.
\end{equation}
We have $\grad_x B(u, w; x) \equiv \xop(u, w) \in X$ for all $x \in X$.

Clearly, also, $\linwop$ is an abbreviation for $(u, w; x) \to D_x B(u, w, 0)(x)$, where, just here, we write $D_x$ for the Fréchet derivative with respect to $x$.
Likewise we write $\uop$ to abbreviate $(u, w; x) \to D_u B(0, w, x)(u)$, and $\xuop$ to abbreviate $(u, w; x) \to D_uB_x(0, w, x)(u)$.
If $\wop$ is linear in $u$, then $\uop=\wop$; and if $\wop$ is linear in both $u$ and $x$, then $\xuop = \wop$.

We may now write \eqref{eq:algorithm:problem} as\footnote{If the PDE \eqref{eq:pde-u} does not have a solution $u$ for any $x \in \Dom F \isect \Dom(G \circ K)$, the inner “max” will be infinite, not reached, and technically, therefore, a “sup”.
In this case also \eqref{eq:algorithm:problem} has no solution.
If \eqref{eq:algorithm:problem} has a solution, there must exist some $(x, u)$ for which (any) $w$ reaches the “max”. Likewise, $y$ reaching the corresponding “max” exists for any $x \in \Dom(G \circ K)$ by basic properties of Fenchel conjugates of convex, proper, lower semicontinuous functions.}
\begin{gather}
    \label{eq:algorithm:minmax:0}
    \min_{x,u}\max_{w}~ F(x) + Q(u) + \wop(u, w; x)-\wrhs w + G(Kx)
\shortintertext{or}
    \label{eq:algorithm:minmax}
    \min_{x,u}\max_{w,y}~ F(x) + Q(u) + \wop(u, w; x)-\wrhs w + \iprod{ K x}{y}_Y - G^*(y).
\end{gather}
In terms of $(\optu, \optw, \optx, \opty) \in U \times W \times X \times Y$, subject to a qualification condition, this problem has the necessary first-order optimality conditions
\begin{equation}
    \label{eq:algorithm:oc1}
    \left\{\begin{aligned}
        \wop(\opt u, \tilde w; \opt x)&= \wrhs\tilde w
        &&\text{for all}\quad \tilde w \in W, \\
        \uop(\tilde u, \optw; \opt x) &= -Q'(\optu)\tilde u
        &&\text{for all}\quad \tilde u \in U,\\
        -\xop(\optu, \optw) -  K^{\star} \opt y &\in \partial F(\opt x), \\
         K\opt x &\in \partial G^*(\opt y).
    \end{aligned}\right.
\end{equation}
This is our principal form of optimality conditions for \eqref{eq:algorithm:problem}.

It is easy to see that \eqref{eq:algorithm:oc1} are necessary for $(\opt u, \opt w, \opt x, \opt y)$ to be a saddle point of \eqref{eq:algorithm:minmax}. The next theorem shows, subject to qualification conditions, that \eqref{eq:algorithm:oc1} are also necessary for a solution to  \eqref{eq:algorithm:minmax} (which may not be a saddle point in the non-convex-concave setting).
Note that $w \in W$ is inconsequential in \eqref{eq:algorithm:minmax}.
If one choice forms a part of a solution of the problem, so does any other (or else the problem has no solution at all).
However, $\opt w$ solving \eqref{eq:algorithm:oc1} is more precisely determined.

\begin{theorem}
    \label{thm:algorithm:oc}
    Suppose $(\optu, w, \optx, \opty) \in U \times W \times X \times Y$ solve \eqref{eq:algorithm:minmax}. If, moreover, $\interior \Dom [F + G \circ K] \ne \emptyset$, and, for some $c>0$,
    \begin{subequations}
    \label{eq:algorithm:cq}
    \begin{gather}
        \label{eq:algorithm:cq1}
        \sup_{\norm{(h_x, h_u)}=1} \linwop(\optu, w; h_x) + \uop(h_u, w; \optx) \ge c \norm{w}
        \quad\text{for all}\quad w \in W
        \quad\text{and}
        \\
        \label{eq:algorithm:cq2}
        \uop(\tilde u, w; \optx) = 0 \ \text{for all}\  \tilde u
        \implies
        \linwop(\optu, w; x) = 0
        \ \text{for all}\ x \in \Dom (F + G \circ K),
    \end{gather}
    \end{subequations}
    then \eqref{eq:algorithm:oc1} holds for some $\opt w \in W$.
\end{theorem}

After an affine shift and restriction of $x$ to a subspace, the condition $\interior \Dom [F + G \circ K] \ne \emptyset$ can always be relaxed to the corresponding relative interior being non-empty.
Since the proof of \cref{thm:algorithm:oc} is long and depends on techniques not needed in our main line of work, we relegate it to \cref{sec:oc}.

\begin{example}
    If $W=U$, taking $h_u=w/\norm{w}$ and $h_x=0$, we see that the qualification conditions \eqref{eq:algorithm:cq} hold when $B_u(\freevar,\freevar; \optx)$ is coercive. Similarly, also when $W \ne U$, if the weak coercivity conditions of the Babuška--Lax--Milgram theorem hold for $(w, h_u) \mapsto B_u(h_u, w; \optx)$, then so do \eqref{eq:algorithm:cq}.
\end{example}

The second line of \eqref{eq:algorithm:oc1} is the adjoint PDE, needed for $\grad S(\opt x)^* \grad Q(\optu)$ in \eqref{eq:algorithm:oc0}:

\begin{corollary}
    \label{cor:sol-adj-is-W}
    Suppose \eqref{eq:algorithm:cq} hold for $\opt x = x \in X$, some $w \in W$, and $\opt u = u$ a unique solution to \eqref{eq:pde-u}. Then the
    solution operator $S$ of \eqref{eq:pde-u} satisfies for all $z \in U$ that
    \[
        \grad S(x)^{\star} z = \xop(u, w)
        \quad \text{where} \quad
        u = S(x)
        \quad\text{and}\quad
        \left\{\begin{array}{l}
        w \text{ solves the weak adjoint PDE:}\\
        \uop(\tilde u, w; x) = - \iprod{z}{\tilde u} \text{ for all } \tilde u \in U.
        \end{array}\right.
    \]
\end{corollary}

\begin{proof}
    Take $F \equiv 0$, $K=\Id$, $G \equiv \delta_{\{x\}}$, and $Q = \iprod{z}{\freevar}_U $. Then any solution $(\optu, w, \optx, y)$ to \eqref{eq:algorithm:minmax} has $\opt x = x$.
    Since $G^*(\tilde y)=\iprod{x}{\tilde y}$, any choice of $y$ and $w$ solve \eqref{eq:algorithm:minmax}.
    Therefore, \cref{thm:algorithm:oc} applied to the problem we just constructed shows that
    \begin{equation*}
        \uop(\tilde u, w; x) = -\iprod{z}{\tilde u}_U
            \ \text{for all}\ \tilde u \in U
            \quad\text{and}\quad
        -\xop(u, w) -  y = 0.
    \end{equation*}
    On the other hand, \eqref{eq:algorithm:oc0} reduces to some $y$ satisfying
    $
        - \grad S(x)^{\star} z -  y = 0.
    $
    Comparing these two expressions, we obtain the claim.
\end{proof}

\subsection{Primal-dual proximal splitting: a recap}
\label{sec:algorithm:pdps}

The primal-dual proximal splitting (PDPS) for \eqref{eq:intro:prob-pdps} is based on the optimality conditions
\begin{equation}
    \label{eq:algorithm:pdps-oc}
    \left\{\begin{aligned}
        -  K^{\star} \opt y &\in \partial F(\opt x), \\
         K\opt x &\in \partial G^*(\opt y).
    \end{aligned}\right.
\end{equation}
These are just the last two lines of \eqref{eq:algorithm:oc1} without $\xop$. As derived in \cite{tuomov-proxtest,he2012convergence,clasonvalkonen2020nonsmooth}, the basic (unaccelerated) PDPS solves \eqref{eq:algorithm:pdps-oc} by iteratively solving for each $k \in \N$ the system
\begin{equation}
    \label{eq:algorithm:pdps-implicit}
    \left\{\begin{aligned}
        0 &\in \tau\partial F(\nxt x) + \tau K^{\star}\this y + \nxt x - \this x \\
        0 &\in \sigma\partial G^*(\nxt y) - \sigma K[\nxt x + \omega(\nxt x-\this x)] + \nxt y - \this y,
    \end{aligned}\right.
\end{equation}
where the primal and dual step length parameters $\tau, \sigma>0$ satisfy $\tau\sigma\norm{K} < 1$, and the over-relaxation parameter $\omega=1$.
We can write \eqref{eq:algorithm:pdps-implicit} in explicit form as
\[
    \left\{\begin{aligned}
        \nxt x & \defeq \prox_{\tau F}\bigl(\this x - \tau K^{\star}\this y\bigr),
        \\
        \nxt y & \defeq \prox_{\sigma G^*}\bigl(\this y + \sigma K[\nxt x + \omega(\nxt x-\this x)]\bigr).
    \end{aligned}\right.
\]

\subsection{Algorithm derivation}
\label{sec:algorithm:algorithm}

The derivation of the PDPS and the optimality conditions \eqref{eq:algorithm:oc1} suggest to solve \eqref{eq:algorithm:oc1} by iteratively solving
\begin{equation}
    \label{eq:algorithm:implicit0}
    \left\{\begin{aligned}
        \wop(\nxt u, \freevar; \this x)&= \wrhs,
        \\
        \uop(\freevar, \nxt w; \this x) &= -Q'(\nxt u),
        \\
        0 &\in \tau_k\partial F(\nxt x) + \tau_k \xop(\nxt u, \nxt w) + \tau K^{\star}\this y + \nxt x - \this x \\
        0 &\in \sigma_{k+1}\partial G^*(\nxt y) - \sigma_{k+1} K[\nxt x + \omega_k(\nxt x-\this x)] + \nxt y - \this y.
    \end{aligned}\right.
\end{equation}
We have made the step length and over-relaxation parameters iteration-dependent for acceleration purposes.
The indexing $\tau_k$ and $\sigma_{k+1}$ is off-by-one to maintain the symmetric update rules from \cite{chambolle2011first}.

The method in \eqref{eq:algorithm:implicit0} still requires exact solution of the PDEs.
For some splitting operators $\Gamma_k, \Upsilon_k: U \times W \times X \to \R$, we therefore transform the first two lines into
\begin{subequations}
\label{eq:algorithm:pde-split}
\begin{align}
    \wop(\nxt u, \freevar; \this x) - \Gamma_k(\nxt u - \this u, \freevar; \this x) &= \wrhs
    \quad\text{and}
    \\
    \uop(\freevar, \nxt w; \this x) - \Upsilon_k(\freevar, \nxt w - \thisw; \this x) &= - Q'(\nxt u).
\end{align}
\end{subequations}

\begin{example}[Splitting]
    \label{ex:splitting}
    Let $B(u, w; x)= \iprod{A_x u}{w}$ for symmetric $A_x \in \R^{n \times n}$ on $U=W=\R^n$.
    Take $\Gamma_k(u, w; x)=\iprod{[A_x - N_x] u}{w}$ and $\Upsilon_k=\Gamma_k$ for easily invertible $N_x \in \R^{n \times n}$. With $L=\iprod{b}{\freevar}$, $b \in \R^n$ and $M_x \defeq A_x-N_x$, \eqref{eq:algorithm:pde-split} now reads
    \begin{equation}
        \label{eq:algorithm:splitting-update}
        N_{\this x} \nxt u = b - M_{\this x} \this u
        \quad\text{and}\quad
        N_{\this x} \nxt w = - \grad Q(\nxt u) - M_{\this x} \thisw.
    \end{equation}
    For Jacobi splitting we take $N_{\this x}$ as the diagonal part of $A_{\this x}$, and for Gauss–Seidel splitting as the lower triangle including the diagonal.
    We study these choices further in \cref{sec:examples:splittings}.
\end{example}

Let us introduce the general notation $v=(u, w, x, y)$ as well as the \term{step length operators} $\Step_k \in \linear(U^* \times W^* \times X \times Y; U^* \times W^* \times X \times Y)$,
\begin{equation}
    \label{eq:algorithm:stepk}
    \Step_k \defeq \diag\begin{pmatrix} \Id_{U^*} & \Id_{W^*} & \tau_k \Id_X & \sigma_{k+1} \Id_Y\end{pmatrix},
\end{equation}
the set-valued operators $ H_k: U \times W \times X \times Y \setto U^* \times W^* \times X \times Y$,
\begin{equation}
    \label{eq:algorithm:hk}
    H_k(v) \defeq
    \begin{pmatrix}
        \wop(u, \freevar; \this x) - \Gamma_k(u - \this u, \freevar; \this x) - \wrhs \\
        \uop(\freevar, w; \this x) - \Upsilon_k(\freevar, w - \thisw; \this x) + Q'(u)
        \\
        \partial F(x) + \xop(u,w) +  K^{\star}y \\
        \partial G^*(y) -  Kx
    \end{pmatrix},
\end{equation}
and the \term{preconditioning operators} $M_k \in \linear(U \times W \times X \times Y; U^* \times W^* \times X \times Y)$,
\begin{equation}
    \label{eq:algorithm:mk}
    M_k \defeq \begin{pmatrix}
        0 \\
        & 0 \\
        & & \Id_X & -\tau_k K^{\star} \\
        & & -\omega_k\sigma_{k+1} K & \Id_Y
    \end{pmatrix}.
\end{equation}
The implicit form of our proposed algorithm for the solution of \cref{eq:algorithm:problem} is then
\begin{equation}
    \label{eq:algorithm:implicit}
    0 \in \Step_k H_k(\nxt v) + M_k(\nxt v-\this v).
\end{equation}
Writing out \eqref{eq:algorithm:implicit} in terms of explicit proximal maps, we obtain \cref{alg:alg}.

\begin{remark}
    The index $k$ for $\Step_k, H_k, M_k$ in \crefrange{eq:algorithm:stepk}{eq:algorithm:implicit} is inconsistent with some of our earlier articles that would use the index $k+1$ similarly to the unknown $\nxt v$.
    We have decided to make this change to keep the notation lighter.
\end{remark}

\begin{algorithm}
	\caption{Primal dual splitting with parallel adaptive PDE solves (PDPAP)}
	\label{alg:alg}
	\begin{algorithmic}[1]
        \Require $F: X \to \extR$, $G^*: Y \to \extR$, Fréchet-differentiable $Q: U \to \R$; $K \in \linear(X; Y)$, $L \in U^*$; and $B: U \times W \times X \to \R$, bilinear in the first two variables, affine in the third, all on Hilbert spaces $X$, $Y$, $U$, and $W$.
        Riesz representation $\xop(u, w)$ of $\linwop(u, w; \freevar)$; see \eqref{eq:algorithm:riesz}.
        For all $k \in \N$, splittings $\Gamma_k, \Upsilon_k: U \times W \times X \to \R$ and step length and over-relaxation parameters $\tau_k,\sigma_{k+1},\omega_k>0$;
        see \cref{thm:convergence:accel} or \ref{thm:convergence:linear}.
        \State Pick an initial iterate $(u^0, w^0, x^0, y^0) \in U \times W \times X \times Y$.
  		\For{$k \in \N$}
            \State\label{step:alg:pde}
            Solve $\nexxt u \in U$ from the split weak PDE
            \[
                \wop(\nxt u, \tilde w; \this x) - \Gamma_k(\nxt u - \this u, \tilde w; \this x) = \wrhs \tilde w
                \quad\text{for all}\quad\tilde w \in W.
            \]
            \State\label{step:alg:adjoint-pde}
            Solve $\nexxt w \in W$ from the split weak adjoint PDE
            \[
                \uop(\tilde u, \nxt w; \this x) - \Upsilon_k(\tilde u, \nxt w - \thisw; \this x) = - Q'(\nxt u)\tilde u
                \quad\text{for all}\quad\tilde u \in U.
            \]

            \State
            $
                \nxt x \defeq \prox_{\tau_k F}\bigl(\this x - \tau_k \xop(\nxt u, \nxt w) - \tau_k K^{\star}\this y\bigr)
            $
            \State
            $
                \nxxt x \defeq \nxt x + \omega_k(\nxt x - \this x)
            $
            \State
            $
                \nxt y \defeq \prox_{\sigma_{k+1} G^*}\bigl(\this y + \sigma_{k+1} K\nxxt x\bigr)
            $
        \EndFor
    \end{algorithmic}
\end{algorithm}

\section{Convergence}
\label{sec:convergence}

We now treat the convergence of \cref{alg:alg}.
Following \cite{tuomov-proxtest,clasonvalkonen2020nonsmooth} we “test” its implicit form \eqref{eq:algorithm:implicit} by applying on both sides the linear functional $\dualprod{Z_k\freevar}{\nexxt v - \opt v}$. Here $Z_k$ is a convergence rate encoding “testing operator” (\cref{sec:convergence:testing}).
A simple argument involving the three-point identity \eqref{eq:pythagoras} and a growth estimate  for $H_k$ then yields in \cref{sec:convergence:growth} a Féjer-type monotonicity estimate in terms of iteration-dependent norms. This establishes in \cref{sec:convergence:main} global convergence subject to a growth condition.
We start with assumptions.

\subsection{The main assumptions}
\label{sec:convergence:assumptions}

We start with our main structural assumption. Further central conditions related to the PDE constraint will follow in \cref{assump:convergence:splitting}, and through its verification for specific linear system solvers in \cref{sec:examples:splittings}.

\begin{assumption}[Structure]
    \label{assump:convergence:structural}
    On Hilbert spaces $X$, $Y$, $U$, and $W$, we are given convex, proper, and lower semicontinuous $F: X \to \extR$, $G^*: Y \to \extR$, and $Q: U \to \R$ with $Q$ Fréchet differentiable, as well as $K \in \linear(X; Y)$, $L \in U^*$, and $B: U \times W \times X \to \R$ affine-linear-affine.
    We assume:
    \begin{enumerate}[label=(\roman*)]
        \item
        \label{item:convergence:structural:convexity}
        $F$ and $G$ are (strongly) convex with factors $ \gamma_F, \gamma_{G^*} \geq 0 $. With $K$ they satisfy the condition \eqref{eq:algorithm:sumrule-ok} for the subdifferential sum and chain rules to be exact.
        \item\label{item:convergence:structural:subspace}
        For all $x \in  \Dom F$, there exist solutions $(u, w) \in U \times W$ to the PDE $\wop(u,  \freevar; x) = \wrhs$ and the adjoint PDE $\uop(\freevar, w; x) = -Q'(u)$.
    \end{enumerate}
    We then fix a solution $\opt v=(\opt u, \opt w, \opt x, \opt y) \in U \times W \times X \times Y$ to \eqref{eq:algorithm:oc1} and assume that:
    \begin{enumerate}[resume*]
        \item\label{item:convergence:structural:bound-sup}
        For some $\SupNorm(\opt u), \SupNorm(\opt w) \ge 0$, for all $(u, w) \in U \times W$ and $x \in \Dom F$, we have
        \begin{align*}
            \xuop(u, \opt w; x - \opt x) \le \sqrt{\SupNorm(\bar w)}\norm{u}_U\norm{x - \optx}_X
            \quad\text{and}\quad
            \linwop(\opt u, w; x - \optx) \le \sqrt{\SupNorm(\bar u)}\norm{w}_W\norm{x - \optx}_X.
        \end{align*}

        \item\label{item:convergence:structural:bound}
        For some $C_x \ge 0$, for all $(u, w) \in U \times W$ and $x \in \Dom F$ we have the bound
        \[
            \xuop(u, w; x-\opt x)
            \le C_x \norm{u}_U\norm{w}_W.
        \]
    \end{enumerate}
\end{assumption}

\begin{remark}
    \label{rem:convergence:structural}
    Part \cref{item:convergence:structural:convexity} is easy to check.
    In general, \cref{item:convergence:structural:bound} requires $\Dom F$ to be bounded with respect to an $\infty$-norm with $\linwop(u, w, x) \le C \norm{u}_U\norm{w}_W\norm{x}_\infty$ for some $C>0$. Then $C_x= \sup_{x \in \Dom F} C \norm{x}_\infty$.
    If $\linwop$ is independent of $u$, i.e., for linear PDEs, both $C_x=0$ and $\SupNorm(\bar w)=0$, while $\SupNorm(\bar u)$ is a constant independent of $\opt u$.
    We study \crefrange{item:convergence:structural:subspace}{item:convergence:structural:bound} further in \cref{sec:examples:W}.
\end{remark}

The next assumption encodes our conditions on the PDE splittings.

\begin{assumption}[Splitting]
    \label{assump:convergence:splitting}
    Let \cref{assump:convergence:structural} hold.
    For $k \in \N$, for which this assumption is to hold, we are given splitting operators $\Gamma_k, \Upsilon_k: U \times W \times X \to \R$ and $\this v=(\this u, \this w, \this x, \this y) \in U \times W \times X \times Y$ such that:
    \begin{enumerate}[label=(\roman*)]
        \item\label{item:convergence:splitting:linear}
        $\Gamma_k$ is linear in the second argument, $\Upsilon_k$ in the first.

        \item\label{item:convergence:splitting:existence}
        There exist solutions $\nextu$ and $\nextw$ to the split equations \eqref{eq:algorithm:pde-split}.

        \item\label{item:convergence:splitting:principal}
        For some $\gamma_B > 0$ and $C_Q, \pi_u, \pi_w \ge 0$, we have
        \begin{align*}
            \norm{\thisu-\optu}_U^2
            &
            \ge
            \gamma_B \norm{\nxt u-\optu}_U^2
            - \pi_u \norm{\this x - \opt x}_X^2
            \quad\text{and}
            \\
            \norm{\thisw-\optw}_W^2
            &
            \ge
            \gamma_B \norm{\nxt w-\optw}_W^2
            - C_Q \norm{\nxt u-\optu}_U^2
            - \pi_w \norm{\this x - \opt x}_X^2.
        \end{align*}
    \end{enumerate}
\end{assumption}

We verify the assumption for standard splittings in \cref{sec:examples:splittings}.
The verification will introduce the assumption that $Q'$ be Lipschitz. The Lipschitz factor then appears in $C_Q$, justifying the $Q$-subscript notation.
Generally $\pi_u$ and $\pi_w$ model the $x$-sensitivity of $\wop$ and $\uop$.
For linear PDEs, such as \cref{ex:algorithm:b-idea-linear}, $\uop$ does not depend on $x$. In that case most iterative solvers for the adjoint PDE would also be independent of $x$ and have $\pi_w=0$.
The factor $\gamma_B$ relates to the contractivity of the iterative solver.

The next, final, assumption introduces \term{testing parameters} that encode convergence rates and restrict the \term{step length parameters} in the standard primal-dual component of our method.
It has no difference to the treatment of the PDPS in \cite{tuomov-proxtest,clasonvalkonen2020nonsmooth}.
Dependent on whether both, one, or none of $\tilde\gamma_F>0$ and $\tilde\gamma_{G^*}>0$, the parameters can be chosen to yield varying modes and rates of convergence.

\begin{assumption}[Primal-dual parameters]
    \label{assump:convergence:testing}
    Let \cref{assump:convergence:structural} hold.
    For all $k \in \N$, the \term{testing parameters} $\varphi_k, \psi_k > 0$, \term{step length parameters} $\tau_k, \sigma_k > 0$, and the \term{over-relaxation parameter} $\omega_k \in (0, 1]$ satisfy for some $\tilde \gamma_F \in [0, \gamma_F]$ and $\tilde \gamma_{G^*} \in [0, \gamma_{G^*}]$, and $\kappa \in (0, 1)$ that
    \begin{align*}
        \phi_{k+1} & = \phi_k(1+2\tilde\gamma_{F}\tau_k),
        &
        \psi_{k+1} & = \psi_k(1+2\tilde\gamma_{G^*}\sigma_k),
        \\
        \eta_k & \defeq \phi_k \tau_k = \psi_k\sigma_k,
        &
        \omega_k &= \inv\eta_{k+1}\eta_{k},
        \quad\text{and}
        &
        \kappa & \ge \frac{\tau_k\sigma_k}{1+2\tilde\gamma_{G^*}\sigma_k}\norm{K}^2.
    \end{align*}
\end{assumption}

\subsection{The testing operator}
\label{sec:convergence:testing}

To complement the primal-dual testing parameters in \cref{assump:convergence:testing}, we introduce testing parameters $\lambda_k,\theta_k>0$ corresponding to the PDE updates in our method; the first two lines of \eqref{eq:algorithm:implicit}.
We combine all of them into the \term{testing operator} $Z_k \in \linear(U^* \times W^* \times X \times Y;  U^* \times W^* \times X^{*} \times Y^{*})$ defined by
\begin{equation}
    \label{eq:convergence:zk}
    Z_k \defeq \diag\begin{pmatrix}\lambda_k \Id & \theta_k \Id & \varphi_k \Inj_X & \psi_{k+1} \Inj_Y\end{pmatrix}.
\end{equation}
Recalling $M_k$ and $Z_k$ from \eqref{eq:algorithm:mk} and \eqref{eq:convergence:zk}, thanks to \cref{assump:convergence:testing}, we have
\begin{equation}
    \label{eq:convergence:zm-expansion}
    Z_kM_k = \begin{pmatrix}
        0 \\
        & 0 \\
        & & \varphi_k\Inj_X & -\eta_k \Inj_X K^{\star} \\
        & & -\eta_k \Inj_Y K & \psi_{k+1}\Inj_Y
    \end{pmatrix}.
\end{equation}
Therefore,
\begin{equation}
    \label{eq:convergence:z-xi-def}
    Z_k(M_k+\Xi_k)=Z_{k+1}M_{k+1}+D_{k+1}
\end{equation}
for skew-symmetric
\[
    D_{k+1} \defeq \begin{pmatrix}
        0 \\
        & 0 \\
        & & 0 & (\eta_{k+1} + \eta_k) \Inj_X K^{\star} \\
        & & -(\eta_{k+1} + \eta_k) \Inj_Y K & 0
    \end{pmatrix}
\]
and $\Xi_k \in \linear(U \times W \times X \times Y;  U^* \times W^* \times X^{*} \times Y^{*})$ satisfying
\begin{equation}
    \label{eq:convergence:z-xi-expansion}
    Z_k\Xi_k = \begin{pmatrix}
        0 \\
        & 0 \\
        & & 2\eta_k\tilde\gamma_{F}\Inj_X & 2\eta_k \Inj_X K^{\star} \\
        & & -2\eta_{k+1} \Inj_Y K & 2\eta_{k+1}\tilde\gamma_{G^*}\Inj_Y
    \end{pmatrix}.
\end{equation}

\Cref{assump:convergence:testing} ensures $Z_kM_k $ to be positive semi-definite.
The proof is exactly as for the PDPS, see, e.g., \cite{clasonvalkonen2020nonsmooth}, but we include it for completeness.

\begin{lemma}
    \label{lemma:convergence:zm-lower-bound}
    Let $k \in \N$ and suppose \cref{assump:convergence:testing} holds.
    Then
    \[
        Z_kM_k \geq\operatorname{diag}\left(
            0,
            0,
            \phi_k (1-\kappa)\Inj_X,
            \psi_{k+1}\epsilon\Inj_Y
        \right)
        \ge 0
        \quad\text{for}\quad
        \epsilon \defeq 1 - \frac{\tau_k\sigma_k}{\kappa(1+2\tilde\gamma_{G^*}\sigma_k)} \norm{K}^2 > 0.
    \]
\end{lemma}

\begin{proof}
    By Young's inequality, for any $v=(u,w,x,y)$,
    \begin{align*}
        \dualprod{Z_k M_k v}{v} &= \varphi_k\norm{x}_X^2 + \psi_{k+1}\norm{y}_{Y}^2 - 2\eta_k\iprodx{x, K^{\star} y}_X
        \\
        &
        \geq \varphi_k(1-\kappa)\norm{x}_X^2 + \psi_{k+1}\norm{y}_Y^2-
        \inv\kappa\varphi_k\tau_k^2\norm{K^{\star}y}_X^2.
    \end{align*}
    Since $\varphi_k\tau_k^2=\eta_k\tau_k=\psi_k\sigma_k\tau_k=\psi_{k+1}\sigma_k\tau_k/(1+2\tilde\gamma_{G^*}\sigma_k)$, the claim follows.
\end{proof}

\subsection{Growth estimates and monotonicity}
\label{sec:convergence:growth}

We start by deriving a three-point monotonicity estimate for $H_k$.
This demands the somewhat strict bounds \eqref{eq:convergence:balance0}.

\begin{lemma}
    \label{lem:ineq-Hk-norm-1}
    Let $k \in \N$. Suppose \cref{assump:convergence:testing,assump:convergence:structural,assump:convergence:splitting} hold and
    \begin{subequations}
        \label{eq:convergence:balance0}
        \begin{align}
            \label{eq:convergence:balance0:gammaf}
            \gamma_F
            &
            \ge
            \tilde\gamma_F
            + \epsilon_u + \epsilon_w
            + \frac{\lambda_{k+1}\pi_u + \theta_{k+1}\pi_w}{\eta_k},
            \\
            \label{eq:convergence:balance0:gammag}
            \gamma_{G^*}
            &
            \ge
            \tilde\gamma_{G^*},
            \\
            \label{eq:convergence:balance0:gammab}
            \gamma_B
            &
            \ge
            \frac{\lambda_{k+1}}{\lambda_k}
            + \frac{\theta_k}{\lambda_k}C_Q
            + \frac{\eta_k\SupNorm(\optw)}{4\epsilon_w\lambda_k}
            + \frac{C_x \mu \eta_k}{2\lambda_k},
            \quad\text{and}
            \\
            \label{eq:convergence:balance0:gammab2}
            \gamma_B
            &
            \ge
            \frac{\theta_{k+1}}{\theta_k}
            + \frac{\eta_k\SupNorm(\opt u)}{4\epsilon_u\theta_k}
            + \frac{C_x\eta_k}{2 \mu \theta_k}
        \end{align}
    \end{subequations}
    for some $\epsilon_u,\epsilon_w,\mu > 0$.
    Then $H_k$ defined in \eqref{eq:algorithm:hk} satisfies
    \begin{equation}
        \label{eq:ineq-Hk-norm-1}
        \begin{aligned}[t]
            \dualprod{Z_k \Step_k H_k(\nxt v)}{\nxt v-\opt v}
            &
            \ge
            \frac12\norm{\nxt v-\opt v}_{Z_k\Xi_k}^2
            \\
            \MoveEqLeft[-1]
            + (\lambda_{k+1}\pi_u + \theta_{k+1}\pi_w) \norm{\nxt x-\opt x}_X^2
            - (\lambda_k\pi_u + \theta_k\pi_w) \norm{\this x-\opt x}_X^2
            \\
            \MoveEqLeft[-1]
            + \lambda_{k+1} \norm{\nxt u-\opt u}_U^2
            - \lambda_k \norm{\thisu-\optu}_U^2
            \\
            \MoveEqLeft[-1]
            + \theta_{k+1} \norm{\nxt w-\optw}_W^2
            - \theta_k \norm{\thisw-\optw}_W^2.
        \end{aligned}
    \end{equation}
\end{lemma}

\begin{proof}
    For brevity we denote $v=(u, w, x, y) \defeq \nxt v$.
    Recall that $\opt v=(\opt u, \opt w, \opt x, \opt y)$ satisfies by \cref{assump:convergence:structural} the optimality conditions \eqref{eq:algorithm:oc1}.
    Since \cref{alg:alg} guarantees the first two lines of $H_k$ to be zero through the choice of $M_k$ in \eqref{eq:algorithm:mk}, introducing $q_F \defeq -\xop(\optu, \optw) -  K^{\star} \opt y \in \partial F(\opt x)$ we expand
    \begin{align*}
        \dualprod{Z_k \Step_k H_k(v)}{v-\opt v}
        &
        =
        \eta_k\iprod{\partial F(x) + \xop(u,w) +  K^{\star} y}{x-\opt x}_X
        + \eta_{k+1}\iprod{\partial G^*(y) -  K x}{y - \opt y}_{Y}
        \\
        &
        =
        \eta_k\iprod{\partial F(x) - q_F}{x-\opt x}_X
        + \eta_k\iprod{\xop(u,w) - \xop(\optu, \optw)}{x-\opt x}_X
        \\
        \MoveEqLeft[-1]
        +\eta_{k+1}\iprod{\partial G^*(y) -  K\opt x}{y- \opt y}_Y
        + (\eta_k-\eta_{k+1})\iprod{ K(x-\opt x)}{y - \opt y}_Y.
    \end{align*}
    Using \eqref{eq:convergence:z-xi-expansion} we also have
    \begin{align*}
        \frac12\norm{v-\opt v}_{Z_k\Xi_k}^2
        &
        = \eta_k\tilde\gamma_{F}\norm{x-\opt x}_X^2
        + (\eta_k-\eta_{k+1})\iprodx{ K(x-\opt x),y - \opt y}_{Y}
        + \eta_{k+1}\tilde\gamma_{G^*}\norm{|y-\opt y}_{Y}^2.
    \end{align*}
    We now use the (strong) monotonicity of $F$ and $G^*$ with constants $ \gamma_F$ and $\gamma_{G^*}$ contained \cref{assump:convergence:structural}\,\cref{item:convergence:structural:convexity}, as well as the splitting inequality \cref{assump:convergence:splitting}\,\cref{item:convergence:splitting:principal}.
    Thus
    \begin{equation}
        \label{eq:ineq-Hk-norm-1:final:0}
        \begin{aligned}[t]
        \dualprod{Z_k\Step_k H_k(v)}{v-\opt v}
        &
        \ge
        \frac{1}{2}\norm{v-\opt v}_{Z_k\Xi_k}^2
        + \eta_k(\gamma_F-\tilde\gamma_F)\norm{x-\opt x}_X^2
        - (\lambda_k\pi_u + \theta_k\pi_w)\norm{\this x - \opt x}_X^2
        \\
        \MoveEqLeft[-1]
        + \eta_{k+1}(\gamma_{G^*}-\tilde\gamma_{G^*})\norm{y-\opt y}^2_{Y}
        + \eta_k\iprod{\xop(u, w)-\xop(\optu, \optw)}{x-\opt x}_X
        \\
        \MoveEqLeft[-1]
        + (\lambda_k \gamma_B-\theta_k C_Q)\norm{u-\opt u}_U^2
        - \lambda_k \norm{\thisu-\optu}_U^2
        \\
        \MoveEqLeft[-1]
        + \theta_k \gamma_B\norm{w-\optw}_W^2
        - \theta_k \norm{\thisw-\optw}_W^2.
        \end{aligned}
    \end{equation}
    The Riesz equivalence \eqref{eq:algorithm:riesz}, affine-linear-linear structure of $\linwop$, \cref{assump:convergence:structural}\,\cref{item:convergence:structural:bound-sup} and \cref{item:convergence:structural:bound}, and Young's inequality give
    \begin{equation}
        \label{eq:ineq-Hk-norm-1:final:2}
        \begin{aligned}[t]
        \eta_k\iprod{\xop(u, w)-\xop(\optu, \optw)}{x-\opt x}_X
        &
        =
        \eta_k\linwop(u, w, x-\optx) - \eta_k\linwop(\optu, \optw, x-\optx)
        \\
        \MoveEqLeft[13]%
        =
        \eta_k\linwop(u, w, x-\optx)
        + \eta_k\linwop(\optu, w - \optw, x-\optx)
        - \eta_k\linwop(\optu, w, x-\opt x)
        \\
        \MoveEqLeft[13]%
        =
        \eta_k \xuop(u - \opt u, w - \optw; x - \opt x)
         + \eta_k \linwop(\opt u, w - \optw; x - \opt x)
         + \eta_k \xuop(u - \opt u, \optw; x - \opt x)
        \\
        \MoveEqLeft[13]
        \ge
        - \eta_k\left(\frac{\SupNorm(\optu)}{4\epsilon_u}+\frac{C_x \mu }{2}\right)\norm{w-\optw}_W^2
        - \eta_k\left(\frac{\SupNorm(\optw)}{4\epsilon_w}+\frac{C_x}{2 \mu }\right)\norm{u - \opt u}_U^2
        - \eta_k(\epsilon_u+\epsilon_w)\norm{x - \opt x}_X^2
        \end{aligned}
    \end{equation}
    Combining \cref{eq:ineq-Hk-norm-1:final:0,eq:ineq-Hk-norm-1:final:2}, we obtain
    \begin{align*}
        \dualprod{Z_k \Step_k H_k(v)}{v-\opt v}
        &
        \ge
        \frac12\norm{v-\opt v}_{Z_k\Xi_k}^2
        + \eta_{k+1}(\gamma_{G^*}-\tilde\gamma_{G^*})\norm{y-\opt y}^2_{Y}
        \\
        \MoveEqLeft[-1]
        + \eta_k(\gamma_F-\tilde\gamma_F-\epsilon_u-\epsilon_w)\norm{x-\opt x}_X^2
        - (\lambda_k\pi_u + \theta_k\pi_w) \norm{\this x-\opt x}_X^2
        \\
        \MoveEqLeft[-1]
        - \lambda_k \norm{\thisu-\optu}_U^2
        + \lambda_k\left(
            \gamma_B
            - \frac{\theta_k}{\lambda_k} C_Q
            - \frac{\eta_k\SupNorm(\bar w)}{4\epsilon_w\lambda_k}
            - \frac{C_x \mu \eta_k}{2\lambda_k}
        \right)\norm{u-\opt u}_U^2
        \\
        \MoveEqLeft[-1]
        - \theta_k \norm{\thisw-\optw}_W^2
        + \theta_k\left(
            \gamma_B
            - \frac{\eta_k\SupNorm(\opt u)}{4\epsilon_u\theta_k}
            - \frac{C_x\eta_k}{2\mu \theta_k}
        \right)\norm{w-\optw}_W^2.
    \end{align*}
    The claim now follows by applying \cref{eq:convergence:balance0}.
\end{proof}

We now simplify and interpret \cref{eq:convergence:balance0}.

\begin{lemma}
    \label{lemma:convergence:simplify-assumption}
    Suppose $\gamma_F > \tilde\gamma_F > 0$ as well as $\gamma_{G^*} \ge \tilde\gamma_{G^*} \ge 0$ and that there exists $\omega, t > 0 $ with $\omega \eta_{k+1} \le \eta_k$ for all $k \in \N$, such that
    \begin{equation}
        \label{eq:convergence:balance}
        \gamma_B
        \ge
        \inv\omega
        + tC_Q
        + \frac{2(1 + \inv t)}{\omega(\gamma_F-\tilde\gamma_F)^2}
        \left(
            \SupNorm(\optu)\pi_w
            + t\SupNorm(\optw)\pi_u
            + \frac{1}{2}\sqrt{t\pi_u\pi_w}C_x(\gamma_F-\tilde\gamma_F)
        \right).
    \end{equation}
    Then there exist $\epsilon_u,\epsilon_w,\mu>0$ and, for all $k \in \N$, $\lambda_k,\theta_k>0$ such that \eqref{eq:convergence:balance0} holds. Moreover
    \begin{equation}
        \label{eq:convergence:balance:lambdathetapi}
        \lambda_k\pi_u + \theta_k\pi_w = \eta_k\omega\frac{\gamma_F-\tilde\gamma_F}{2}.
    \end{equation}
\end{lemma}

\begin{proof}
    We take
    \begin{equation}
        \label{eq:convergence:balance:lambdathetadef}
        \lambda_k \defeq \inv t r \inv\pi_u \eta_k
        \quad\text{and}\quad
        \theta_k \defeq r \inv\pi_w \eta_k
        \quad\text{for}\quad
        r \defeq \frac{(\gamma_F-\tilde\gamma_F)\omega}{2(\inv t + 1)}
        \quad\text{and}\quad
        c_k \defeq \frac{\eta_{k+1}}{\eta_k}.
    \end{equation}
    These expressions readily give \eqref{eq:convergence:balance:lambdathetapi}.
    We then take $\mu \defeq (t\pi_u/\pi_w)^{-1/2}$,
    \[
        \epsilon_u
        \defeq
        \frac{\SupNorm(\optu)}{\SupNorm(\optu) + t\SupNorm(\optw)}
        \frac{\gamma_F - \tilde\gamma_F}{2},
        \quad\text{and}\quad
        \epsilon_w
        \defeq
        \frac{t\SupNorm(\optw)}{\SupNorm(\optu) + t\SupNorm(\optw)}
        \frac{\gamma_F - \tilde\gamma_F}{2}.
    \]
    Since both
    \[
        \frac{\lambda_{k+1}\pi_u + \theta_{k+1}\pi_w}{\eta_k}
        =c_kr(\inv t + 1)
        =c_k\omega\frac{\gamma_F-\tilde\gamma_F}{2}
        \le \frac{\gamma_F-\tilde\gamma_F}{2}
    \]
    and
    $
         \epsilon_u + \epsilon_w = (\gamma_F-\tilde\gamma_F)/2,
    $
    \eqref{eq:convergence:balance0:gammaf} is readily verified, while \eqref{eq:convergence:balance0:gammag} we have assumed.
    Inserting $\lambda_k,\theta_k,\eta_k$, and $\mu$, we also rewrite \eqref{eq:convergence:balance0:gammab} and \eqref{eq:convergence:balance0:gammab2} as
    \begin{align*}
        \gamma_B
        \ge
        c_k
        + t C_Q
        + \frac{t\SupNorm(\optw)\pi_u}{4\epsilon_w r}
        + \frac{\sqrt{t\pi_u\pi_w}C_x}{2 r}
        \quad\text{and}
        \quad
        \gamma_B
        \ge
        c_k
        + \frac{\SupNorm(\opt u)\pi_w}{4\epsilon_u r}
        + \frac{\sqrt{t\pi_u\pi_w}C_x}{2 r}.
    \end{align*}
    After also inserting $\epsilon_u, \epsilon_w$, and $r$, and using $\omega c_k \le 1$, these are readily verified by \eqref{eq:convergence:balance}.
\end{proof}

\begin{remark}
    Since $\eta_{k+1} \ge \eta_k$ for convergent algorithms, i.e., $\inv\omega \ge 1$, letting $\omega=1$ and $\tilde\gamma_F=0$ in \eqref{eq:convergence:balance}, we obtain at the solution $(\opt u, \opt w, \opt x, \opt y)$ a fundamental “second order growth” and splitting condition (via $C_Q$, $\pi_u$, and $\pi_w$) that cannot be avoided by step length parameter choices.
\end{remark}

Our convergence proof is based based on the next Féjer-type monotonicity estimate with respect to the iteration-dependent norms $\norm{\freevar}_{Z_k\tilde M_k}$.
Here $\tilde M_k \in \linear(U \times W \times X \times Y; U^* \times W^* \times X \times Y)$ modifies $M_k$ defined in \eqref{eq:algorithm:mk} as
\begin{equation}
	\label{eq:convergence:mkt}
    \tilde M_k \defeq M_k + \diag\begin{pmatrix}
        \Inj_U
        & \Inj_W
        &\inv\phi_k (\lambda_k\pi_u + \theta_k\pi_w) \Id_X
        & 0
    \end{pmatrix}.
\end{equation}
By \eqref{eq:convergence:zm-expansion} and \cref{assump:convergence:testing}, this satisfies
\begin{equation}
    \label{eq:convergence:zktildemk}
    Z_k\tilde M_k = \begin{pmatrix}
        \lambda_k \Inj_U \\
        & \theta_k \Inj_W \\
        & & (\varphi_k+\lambda_k\pi_u + \theta_k\pi_w) \Inj_X & -\eta_k \Inj_X K^{\star} \\
        & & -\eta_k \Inj_Y K & \psi_{k+1}\Inj_Y
    \end{pmatrix}.
\end{equation}

\begin{lemma}
    \label{lemma:convergence:main-estimate}
    Suppose \cref{assump:convergence:structural,assump:convergence:testing} hold as does \cref{assump:convergence:splitting} and \eqref{eq:convergence:balance0} for $k=0,\ldots,N$.
    Given $v^0$, let $v^1,\ldots,v^{N-1}$ be produced by \cref{alg:alg}.
    Then
    \begin{equation}
        \label{eq:convergence:quantitative-fejer}
        \frac12\norm{\nxt v-\opt v}_{Z_{k+1}\tilde M_{k+1}}^2
        + \frac12\norm{\nxt v - \thisv}_{Z_kM_k}^2
        \leq
        \frac12\norm{\this v - \opt v}_{Z_k\tilde M_k}^2
        \quad (k=0,\ldots,N-1)
    \end{equation}
    where all the terms are non-negative.
\end{lemma}

\begin{proof}
    \Cref{lem:ineq-Hk-norm-1} gives the estimate
    \begin{equation}
        \label{eq:convergence:estimate1}
        \begin{aligned}[t]
        \dualprod{Z_k \Step_k H_k(\nxt v)}{\nxt v-\opt v}
        &
        \ge
        \frac12\norm{\nxt v-\opt v}_{Z_k\Xi_k}^2
        \\
        \MoveEqLeft[-1]
        + (\lambda_{k+1}\pi_u +  \theta_{k+1}\pi_w) \norm{\nxt x-\opt x}_X^2
        - (\lambda_k\pi_u + \theta_k\pi_w) \norm{\this x-\opt x}_X^2
        \\
        \MoveEqLeft[-1]
        + \lambda_{k+1} \norm{\nxt u-\opt u}_U^2
        - \lambda_k \norm{\thisu-\optu}_U^2
        \\
        \MoveEqLeft[-1]
        + \theta_{k+1} \norm{\nxt w-\optw}_W^2
        - \theta_k \norm{\thisw-\optw}_W^2
        \\
        &
        =
        \frac12\norm{\nxt v-\opt v}_{Z_{k+1}(\tilde M_{k+1}-M_{k+1})+Z_k\Xi_k}^2
        - \frac12\norm{\this v-\opt v}_{Z_k(\tilde M_k-M_k)}^2.
        \end{aligned}
    \end{equation}
    By the implicit form \eqref{eq:algorithm:implicit} of \cref{alg:alg}, we have
    $-Z_k M_k(\nxt v-\this v) \in Z_k \Step_k H_k(\nxt v)$.
    Thus \eqref{eq:convergence:estimate1} combined with the three-point identity \eqref{eq:pythagoras} for the operator $M=Z_kM_k$ yields
    \begin{align*}
        \frac12\norm{\this v-\opt v}_{Z_k\tilde M_k}^2
        &
        \ge
        \frac12\norm{\nxt v-\opt v}_{Z_{k+1}(\tilde M_{k+1}-M_{k+1})+Z_k(M_k+\Xi_k)}^2
        +\frac12\norm{\nxt v - \thisv}_{Z_kM_k}^2
    \end{align*}
    Therefore \eqref{eq:convergence:quantitative-fejer} follows by applying \eqref{eq:convergence:z-xi-def}, i.e., $Z_k(M_k+\Xi_k)=Z_{k+1}M_{k+1}+D_k$, where the skew symmetric term $D_k$ does not contribute to the norms.
    Finally, we have $Z_k\tilde M_k \ge Z_kM_k \ge 0$ by \cref{lemma:convergence:zm-lower-bound}, proving the non-negativity of all the terms.
\end{proof}

\subsection{Main results}
\label{sec:convergence:main}

We can now state our main convergence theorems.
In terms of assumptions, the only fundamental difference between the accelerated $O(1/N)$ and the linear convergence result is that the latter requires $G^*$ to be strongly convex and the former doesn't. Both require sufficient second order growth in terms of the respective technical conditions \cref{eq:convergence:balance-accel:2} or \cref{eq:convergence:balance-linear:2}. The step length parameters differ.

\begin{theorem}[Accelerated convergence]
    \label{thm:convergence:accel}
    Suppose \cref{assump:convergence:structural,assump:convergence:splitting} hold with $\gamma_F>0$.
    Put $\tilde\gamma_{G^*}=0$ and pick $\tau_0,\sigma_0,\kappa,t>0$ and $0 < \tilde\gamma_F < \gamma_F$ satisfying
    \begin{subequations}
    \label{eq:convergence:balance-accel}
    \begin{align}
        \label{eq:convergence:balance-accel:0}
        1 &> \kappa \ge \tau_0\sigma_0\norm{K}^2
        \quad\text{and}\quad
        \\
        \label{eq:convergence:balance-accel:2}
        \gamma_B
        &
        \ge
        \inv\omega_0
        + tC_Q
        + \frac{2(1 + \inv t)}{\omega_0(\gamma_F-\tilde\gamma_F)^2}
        \left(
            \SupNorm(\optu)\pi_w
            + t\SupNorm(\optw)\pi_u
            + \frac{1}{2}\sqrt{t\pi_u\pi_w}C_x(\gamma_F-\tilde\gamma_F)
        \right),
    \intertext{where $\omega_0$ is defined as part of the update rules}
        \nonumber
        \tau_{k+1} & \defeq \tau_k\omega_k,
        \quad
        \sigma_{k+1} \defeq \sigma_k/\omega_k,
        \quad\text{and}\quad
        \omega_k \defeq 1/\sqrt{1+2\tilde\gamma_F\tau_k}
        \quad (k \in \N).
    \end{align}
    \end{subequations}
    Let $\{\nxt v\}_{k \in \N}$ be generated by \cref{alg:alg} for any $v^0 \in U \times W \times X \times Y$.
    Then $\this x \to \opt x$ in $X$; $\thisu \to \optu$ in $U$; and $\thisw \to \optw$ in $W$, all strongly at the rate $O(1/N)$.
\end{theorem}

\begin{proof}
    We use \cref{lemma:convergence:main-estimate}, whose assumptions we now verify.
    \Cref{assump:convergence:structural,assump:convergence:splitting} we have assumed.
    As shown in \cite{tuomov-proxtest,clasonvalkonen2020nonsmooth}, \cref{assump:convergence:testing} holds with $\psi_k \equiv \inv\sigma_0\tau_0$,$\phi_0=1$, and $\phi_{k+1} \defeq \phi_k/\omega_k^2$. Moreover, $\{\varphi_k\}_{k \in \N}$ grows at the rate $\Omega(k^2)$. Hence
    \[
        \eta_{k+1} = \inv\omega_k\eta_k = \sqrt{1+2\tilde\gamma_F\tau_k}\eta_k \le \inv\omega_0\eta_k
        \quad\text{for}\quad
        \inv\omega_0 = \sqrt{1+2\tilde\gamma_F\tau_0}.
    \]
    Thus \eqref{eq:convergence:balance-accel} verifies \eqref{eq:convergence:balance} so that  \cref{lemma:convergence:simplify-assumption} verifies \eqref{eq:convergence:balance0}.
    Thus we may apply \cref{lemma:convergence:main-estimate}.
    By summing its result over $k=0,\ldots,N-1$, we get
    \begin{equation}
        \label{eq:convergence:summed-estimate}
        \frac12\norm{v^N-\opt v}_{Z_{N}\tilde M_{N}}^2
        \le
        \frac12\norm{v^0 - \opt v}_{Z_0\tilde M_0}^2.
    \end{equation}

    By \eqref{eq:convergence:zm-expansion}, \eqref{eq:convergence:zktildemk}, and \cref{lemma:convergence:zm-lower-bound} we have
    \begin{equation}
        \label{eq:convergence:zktildemk-estim}
        Z_k\tilde M_k \ge Z_k M_k \ge
        \diag\begin{pmatrix}
            \lambda_k \Inj_U &
            \theta_k \Inj_W &
            \phi_k (1-\kappa)\Inj_X &
            \psi_{k+1}\epsilon\Inj_Y
        \end{pmatrix}
        \ge 0.
    \end{equation}
    where $\epsilon \defeq 1-\tau_k\sigma_k\inv\kappa\norm{K}^2= 1-\tau_0\sigma_0\inv\kappa\norm{K}^2>0$ by assumption.
    By \cref{lemma:convergence:simplify-assumption}, $\{\lambda_k\}_{k \in \N}$ and $\{\theta_k\}_{k \in \N}$ grow at the same $\Omega(k^2)$ rate as $\{\phi_k\}_{k \in \N}$.
    Therefore \eqref{eq:convergence:summed-estimate} and \eqref{eq:convergence:zktildemk-estim} establish $\norm{\thisx - \optx}_X^2 \to 0$ as well as $\norm{\thisu-\optu}_U^2$ and $\norm{\thisw - \optw}_W^2 \to 0$, all at the rate $O(1/N^2)$. The claim follows by removing the squares.
\end{proof}

\begin{theorem}[Linear convergence]
    \label{thm:convergence:linear}
    Suppose \cref{assump:convergence:structural,assump:convergence:splitting} hold with both $\gamma_F>0$ and $\gamma_{G^*}>0$.
    Pick $\tau,\kappa,t>0$, $0 < \tilde\gamma_{F} \le \gamma_{F}$, $0 < \tilde\gamma_{G^*} \le \gamma_{G^*}$ satisfying
    \begin{subequations}
    \label{eq:convergence:balance-linear}
    \begin{align}
        \label{eq:convergence:balance-linear:0}
        1 & > \kappa \ge \tau^2\inv{\tilde\gamma_{G^*}}\tilde\gamma_F\norm{K}^2
        \quad\text{and}\quad
        \\
        \label{eq:convergence:balance-linear:2}
        \gamma_B
        &
        \ge
        \inv\omega
        + tC_Q
        + \frac{2(1 + \inv t)}{\omega(\gamma_F-\tilde\gamma_F)^2}
        \left(
            \SupNorm(\optu)\pi_w
            + t\SupNorm(\optw)\pi_u
            + \frac{1}{2}\sqrt{t\pi_u\pi_w}C_x(\gamma_F-\tilde\gamma_F)
        \right)
    \shortintertext{for}
        \nonumber
        \sigma & \defeq \inv{\tilde\gamma_{G^*}}\tilde\gamma_F\tau
        \quad\text{and}\quad
        \omega \defeq 1/(1+2\tilde\gamma_F\tau)=1/(1+\tilde\gamma_{G^*}\sigma).
    \end{align}
    \end{subequations}
    Take $\tau_k \equiv \tau$, $\sigma_k \equiv \sigma$, and $\omega_k \equiv \omega$.
    Let $\{\nxt v\}_{k \in \N}$ be generated by \cref{alg:alg} for any $v^0 \in U \times W \times X \times Y$.
    Then $\this x \to \opt x$ in $X$; $\thisu \to \optu$ in $U$; and $\thisw \to \optw$ in $W$, all strongly at a linear rate.
\end{theorem}

\begin{proof}
    As shown in \cite{tuomov-proxtest,clasonvalkonen2020nonsmooth}, \cref{assump:convergence:testing} is satisfied for $\varphi_0=1$, $\psi_0=\inv\sigma\tau$, $\varphi_{k+1} \defeq \varphi_k/\omega_k$, and $\psi_{k+1} \defeq \psi_k/\omega_k$.
    Moreover, both $\{\phi_k\}_{k \in \N}$ and $\{\psi_k\}_{k \in \N}$ grow exponentially and
    $
        \eta_{k+1} \le \inv\omega\eta_k.
    $
    Thus \eqref{eq:convergence:balance-linear} verifies \eqref{eq:convergence:balance} with $c=\inv\omega$ so that  \cref{lemma:convergence:simplify-assumption} verifies \eqref{eq:convergence:balance0}.
    The rest follows as in the proof of \cref{thm:convergence:accel}.
\end{proof}

\Cref{thm:convergence:accel,thm:convergence:linear} show global convergence, but may require a very constricted $\Dom F$ through the constant $C_x$ in \cref{assump:convergence:structural}\,\cref{item:convergence:structural:bound}. In \cref{sec:localisation} we relax the constant by localizing the convergence.

\begin{remark}[Linear and sufficiently linear PDEs]
    \label{rem:linear-pde-condition}
    For linear PDEs, i.e., when $\linwop$ does not depend on $u$, we have $C_x=0$ and $\SupNorm(\opt w)=0$, as observed in \cref{rem:convergence:structural}.
    Moreover, for typical solvers for the adjoint PDE, we would have $\pi_w=0$, as $\uop$ does not then depend on $x$.
    In that case, by taking $t \downto 0$, \eqref{eq:convergence:balance-accel:2} (and likewise \eqref{eq:convergence:balance-linear:2}) reduces to $\gamma_B > \inv\omega_0$.
    Practically this means that the convergence rate factor $\inv\omega_0$ has to be bounded by the inverse contractivity factor $\gamma_B$ of the linear system solver.
    If $\gamma_B>1$, as we should have, this condition can be satisfied by suitable choices of $\tilde\gamma_F \in (0, \gamma_F]$ and $\tilde\gamma_{G^*}$.
    By extension then, the conditions \eqref{eq:convergence:balance-accel:2} and \eqref{eq:convergence:balance-linear:2} are satisfiable for small $t$ when the PDE is “sufficiently linear”.
\end{remark}

\begin{remark}[Weak convergence]
    \label{rem:convergence:weak}
    It is possible to prove weak convergence when $\omega \equiv 1$ and $\tau \equiv \tau_0$, $\sigma \equiv \sigma_0$ satisfy \eqref{eq:convergence:balance-accel}.
    The proof is based on an extension of Opial's lemma to the quantitative Féjer monotonicity \eqref{eq:convergence:quantitative-fejer}.
    We have not included the proof since it is technical, and does not permit reducing assumptions from those of \cref{thm:convergence:accel,thm:convergence:linear}.
    We refer to \cite{tuomov-nlpdhgm-redo} for the corresponding proof for the NL-PDPS.
\end{remark}

\section{Splittings and partial differential equations}
\label{sec:examples}

We now prove \cref{assump:convergence:structural} and derive explicit expressions for the operator $\xop$ from \eqref{eq:algorithm:riesz}.
We do this in \cref{sec:examples:W} for some sample PDEs.
Then in \cref{sec:examples:splittings} we study the satisfaction of \cref{assump:convergence:splitting} for Gauss–Seidel and Jacobi splitting, as well as a simple infinite-dimensional example without splitting. We briefly discuss a quasi-conjugate gradient splitting to illustrate the generality of our approach.
We conclude with a discussion of the convergence theory and discretisation in \cref{sec:examples:discussion}.

\subsection{Partial differential equations and Riesz representations}
\label{sec:examples:W}

Let $\Sym^d \subset \R^{d \times d}$ stand for the symmetric matrices.
Recall that in \cref{ex:algorithm:b-idea}, to ensure the continuity of $B$, we needed in practise that at least one of the spaces $U$, $W$, or $X$ be finite-dimensional. The same will be the case here. Accordingly, with $\Omega \subset \R^d$ a Lipschitz domain, we take
\begin{subequations}
\label{eq:examples:setup}
\begin{equation}
    \label{eq:examples:W:x-spaces}
    x = (A,c) \in X \defeq X_1 \times X_2
    \quad\text{for subspaces}\quad
    X_1 \subset L^2(\Omega; \Sym^d)
    \quad\text{and}\quad X_2 \subset L^2(\Omega),
\end{equation}
as well as $U \subset H^1(\Omega)$ and $W \subset H_0^1(\Omega) \times H^{1/2}(\partial \Omega)$ such that
\begin{align}
    \label{eq:B-example}
    \wop(u, w; x) & \defeq \linwop(u, w; x) + \constwop(u, w)
    \quad\text{for}\quad u \in U,\, w \in W,\, x \in X
\shortintertext{is continuous, where, writing $w=(w_\Omega, w_\partial)$,}
    \linwop(u, w; x)
    &
    \defeq
    \iprod{\grad u}{A\grad w_\Omega}_{L^2(\Omega)} + \iprod{cu}{w_\Omega}_{L^2(\Omega)}
    \quad\text{and}\
    \\
    \constwop(u, w)
    &
    \defeq
    \iprod{\trace_{\partial\Omega} u}{w_\partial}_{L^2(\partial\Omega)}.
\intertext{Thus $\constwop$ models the nonhomogeneous Dirichlet boundary condition $u=g$ on $\partial \Omega$ for some $g \in H^{-\frac{1}{2}}(\partial \Omega) $. Correspondingly we take for some $L_0 \in H^{-1}(\Omega)$ the right-hand-side}
    Lw & \defeq L_0w_\Omega + \iprod{g}{w_\partial}_{L^2(\partial\Omega)}.
\end{align}
\end{subequations}

The next lemma verifies the PDE components of \cref{assump:convergence:structural}. Afterwards we look at particular choices of $X_1$ and $X_2$.
We could also take $W=H^1(\Omega)$, $w=w_\Omega$, $L=L_0$, and $\constwop=0$ to model Neumann boundary conditions, and the result would still hold.
In the range spaces of $L^p(\Omega; \R^d)$, $W^{1,p}(\Omega)$, and $L^p(\Omega; \R^{d \times d})$, we use the Euclidean norm in $\R^d$ and the spectral norm $\norm{\freevar}_2$ in $\R^{d \times d}$.

\begin{lemma}
    \label{lemma:examples:w:assumptions}
    Assume \cref{eq:examples:setup} and that $\Dom F \subset L^\infty(\Omega; \R^{d \times d}) \times L^\infty(\Omega)$.
    Then:
    \begin{enumerate}[label=(\roman*$\,'\!$), ref=(\roman*$'$), start=2]
        \item\label{item:examples:w:assumptions:subspace}
        \cref{assump:convergence:structural}\,\ref{item:convergence:structural:subspace} holds if there exists $ \lambda \in (0,1) $ such that
        \[
            A(\xi) \ge \lambda\,\Id
            \quad\text{and}\quad
            \abs{c(\xi)} \ge \lambda
            \quad\text{for all}\quad
            \xi \in \Omega
            \quad\text{and}\quad
            (A,c) \in (X_1\times X_2)\isect\Dom F.
        \]
    \end{enumerate}
    Suppose then that \eqref{eq:algorithm:oc1} is solved by $\opt v=(\opt u,\opt w,\opt x,\opt y)$  with $\optx=(\opt A, \opt c) \in \Dom F \subset  (X_1 \times X_2) $, $\optu \in H^1(\Omega) $, $\optw = (\optw_\Omega, \optw_\partial) \in H_0^1(\Omega)  \times H^{1/2}(\partial \Omega)$. If $ \|\optu\|_{W^{1,\infty}(\Omega)}, \|\optw\|_{W^{1,\infty}(\Omega)} < \infty $, and $\opty \in Y$ for a Hilbert space $Y$, then also:
    \begin{enumerate}[resume*]
        \item\label{item:examples:w:assumptions:bound-sup}
        \cref{assump:convergence:structural}\,\ref{item:convergence:structural:bound-sup} holds with $ \SupNorm(\optu) = \norm{\optu}_{W^{1,\infty}(\Omega)}^2$ and $ \SupNorm(\optw) = \norm{\optw_\Omega}_{W^{1,\infty}(\Omega)}^2$.
        \item\label{item:examples:w:assumptions:bound}
        \cref{assump:convergence:structural}\,\ref{item:convergence:structural:bound} holds with
        \[
            C_x = \sup_{(A, c) \in \Dom F}~ \norm{A-\opt{A}}_{L^\infty(\Omega; \R^{d \times d})} + \norm{c-\opt{c}}_{L^\infty(\Omega)}.
        \]
    \end{enumerate}
\end{lemma}

\begin{remark}
    \label{rem:examples:w:bounds}
	On bounded $ \Omega $ the condition $ \|\optu\|_{W^{1,\infty}(\Omega)} < \infty $ is stronger than $ \optu \in H^1(\Omega) $. We include both to emphasise that the latter defines the Hilbert space structure and topology that we generally work with, while the former is a technical restriction that arises from our proofs. Under appropriate smoothness conditions on $\optx$, the boundary of $\Omega$, as well as the boundary data, standard elliptic theory proves that $\optu \in H^1(\Omega)$ is a classical solution, hence Lipschitz and $W^{1,\infty}(\Omega)$ on the whole domain; see, e.g., \cite{evans1998pde}.
\end{remark}

\begin{proof}
    For \ref{item:examples:w:assumptions:subspace}, we identify $ g \in H^{-1/2}(\partial\Omega) $ with $ \hat g \in H^{1/2}(\partial\Omega) $ by the Riesz mapping and fix $ \hat{u} \in H^1(\Omega) $ with $ \trace_{\partial\Omega}\hat{u} = \hat g $.
    This is possible by the definition of $ H^{1/2}(\partial\Omega) $.
    By the Lax--Milgram lemma there is then a unique solution $v \in H_0^1(\Omega)$ to
    \[
        \iprod{\grad v}{A\grad w_\Omega}_{L^2(\Omega)} + \iprod{cv}{w_\Omega}_{L^2(\Omega)} = L_0w_\Omega - \linwop(\hat u,w_\Omega;x) \quad\text{for all}\quad w_\Omega \in H_0^1(\Omega),
    \]
    Now $ u = v + \hat{u} $ satisfies $ B(u,w;x) = Lw $ and is independent of the choice of $ \hat{u} $.
    Analogously we prove the existence of a solution to the adjoint equation.

    To prove \cref{item:examples:w:assumptions:bound}, pick arbitrary $u \in H^1(\Omega)$, $w=(w_\Omega, w_\partial) \in H_0^1(\Omega) \times H^{1/2}(\partial\Omega)$, and $x = (A,c) \in (X_1 \times X_2) \isect \Dom F$.
    Hölder's inequality and the symmetry of $A(\xi)$ give
    \[
        \begin{aligned}[t]
        \iprod{\grad u}{A\grad w_\Omega}_{L^2(\Omega)}
        &
        \le
        \norm{\grad w_\Omega}_{L^2(\Omega; \R^d)}
        \left(\int_\Omega \norm{A(\xi)\grad u(\xi)}_2^2 \d \xi\right)^{1/2}
        \\
        &
        \le
        \norm{\grad w_\Omega}_{L^2(\Omega; \R^d)}\norm{A}_{L^\infty(\Omega; \R^{d \times d})}\norm{\grad u}_{L^2(\Omega)}.
        \end{aligned}
    \]
    Therefore, as claimed
    \begin{align*}
        \linwop(u,w;x-\optx)
        &
        \le \norm{A-\opt{A}}_{L^\infty(\Omega; \R^{d \times d})}\norm{\grad u}_{L^2(\Omega; \R^d)}\norm{\grad w_\Omega}_{L^2(\Omega; \R^d)}
        \\
        \MoveEqLeft[-1]
        + \norm{c-\opt{c}}_{L^\infty(\Omega)}\norm{u}_{L^2(\Omega)}\norm{w_\Omega}_{L^2(\Omega)}
        \\
        &
        \le \bigl(\norm{A-\opt{A}}_{L^\infty(\Omega; \R^{d \times d})}
        + \norm{c-\opt{c}}_{L^\infty(\Omega)}\bigr)\norm{u}_{H^1(\Omega)}\norm{w_\Omega}_{H^1(\Omega)}
        \\
        &
        \le
        C_x\norm{u}_{H^1(\Omega)}\norm{w_\Omega}_{H^1(\Omega)}.
    \end{align*}

    For \ref{item:examples:w:assumptions:bound-sup}, using Hölder's twice inequality and the symmetry of $A(\xi)$, we estimate
    \[
        \begin{aligned}[t]
        \iprod{\grad u}{A\grad w_\Omega}_{L^2(\Omega)}
        &
        \le
        \norm{\grad w_\Omega}_{L^\infty(\Omega; \R^d)}
        \int_\Omega \norm{A(\xi)\grad u(\xi)}_2 \d \xi
        \\
        &
        \le
        \norm{\grad w_\Omega}_{L^\infty(\Omega; \R^d)}\norm{A}_{L^2(\Omega; \R^{d \times d})}\norm{\grad u}_{L^2(\Omega)}.
        \end{aligned}
    \]
    Hence
    \[
        \begin{aligned}[t]
        \linwop(u,\optw;x)
        &
        \leq
        \norm{A}_{L^2(\Omega; \R^{d\times d})}\norm{\grad u}_{L^2(\Omega; \R^d)}\norm{\grad \optw_\Omega}_{L^\infty(\Omega; \R^d)}
        \\
        \MoveEqLeft[-1]
        + \norm{c}_{L^2(\Omega)}\norm{u}_{L^2(\Omega)}\norm{\optw_\Omega}_{L^\infty(\Omega)}
        \\
        &
        \leq
        \bigl(
            \norm{\grad \optw_\Omega}_{L^\infty(\Omega; \R^d)} + \norm{\optw_\Omega}_{L^\infty(\Omega)}
        \bigr)
        \norm{u}_{H^1(\Omega)}\bigl(
            \norm{A}_{L^2(\Omega; \R^{d\times d})} + \norm{c}_{L^2}
        \bigr)
        \\
        &
        = \norm{\optw_\Omega}_{W^{1,\infty}(\Omega)}\norm{u}_{H^1(\Omega)}\norm{x}_X.
        \end{aligned}
    \]
    Thus we may take as claimed $\SupNorm(\optw) = \norm{\optw_\Omega}_{W^{1,\infty}(\Omega)}^2$, and analogously $\SupNorm(\optu) = \norm{\optu}_{W^{1,\infty}(\Omega)}^2$.
\end{proof}

To describe $\xop$ we denote the double dot product and the outer product by
\[
    A \outerprod \tilde A = \sum_{ij}A_{ij}\tilde A_{ij},
    \quad\text{and}\quad
    v \otimes w = vw^T
    \quad\text{for}\quad
    A,\tilde A \in \R^{d\times d}
    \quad\text{and}\quad
    v,w \in \R^d
\]
Observe the identity $ v^TAw = A \outerprod (v \otimes w)$.

\begin{example}[General case] \label{ex:W:general}
    In the fully general case, formally and without regard for the solvability of the PDE \eqref{eq:pde-u}, we equip $ X_1 = L^2(\Omega; \R^{d\times d}) $ with the inner product $ \iprod{A_1}{A_2}_{X_1} := \int_\Omega A_1(\xi) \outerprod A_2(\xi) \d\xi $ and $X_2 = L^2(\Omega; \R)$  with the standard inner product in $L^2(\Omega; \R)$. Then for all $u \in U$, $w \in W$, and $(d, h) \in X_1 \times X_2$, we have
    \[
        \linwop(u,w; (d,h))
        = \iprod{\grad u}{d\grad w}_{L^2(\Omega)} + \iprod{hu}{w}_{L^2(\Omega)} = \iprod{\grad u \otimes \grad w}{d}_{X_1} + \iprod{uw}{h}_{X_2}.
    \]
    Therefore the Riesz representation $\xop$ has pointwise in $\Omega$ the expression
    \[
        \xop(u,w) = \begin{pmatrix}
            \grad u \otimes \grad w \\
            uw
        \end{pmatrix}.
    \]
    The constant $C_x$ is as provided by \cref{lemma:examples:w:assumptions}.
\end{example}

\begin{example}[Scalar function diffusion coefficient]
    \label{ex:W:aI}
    Let then $X_1 \defeq \{ \xi \mapsto a(\xi)\Id \mid a \in L^2(\Omega) \}$.
    $X_1$ is isometrically isomorphic with $L^2(\Omega)$ since the spectral norm
    $\norm{a(\xi)\Id}_2 = \abs{a(\xi)}$. We may therefore identify $X_1$ and $L^2(\Omega)$.
    We also observe that the term $\iprod{\grad u}{A\grad w}_{L^2(\Omega)} = \iprod{a}{\grad u\cdot \grad w}_{X_1} $. Hence, pointwise in $\Omega$,
    \[
        \xop(u,w) = \begin{pmatrix}
            \grad u \cdot \grad w \\
            uw
        \end{pmatrix}.
    \]
    According to \cref{lemma:examples:w:assumptions}, the constant
    \[
        C_x = \sup_{(a, c) \in \Dom F}~\norm{a-\opt a}_{L^\infty(\Omega)} + \norm{c-\opt{c}}_{L^\infty(\Omega)}.
    \]
\end{example}

\begin{example}[Spatially uniform coefficients] \label{ex:W:scalar}
    Let $X_1 \defeq \{ \xi \mapsto \tilde A \mid \tilde A \in \Sym^d \} \subset L^2(\Omega; \Sym^d)$ and $X_2 \defeq \{ \xi \mapsto \tilde c \mid \tilde c \in \R \} \subset L^2(\Omega)$ consist of constant functions $A: \xi \mapsto \tilde A$ and $c: \xi \mapsto \tilde c$ on the bounded domain $\Omega$.
    Then $\norm{x}_{X_1 \times X_2} = \abs{\Omega}^{1/2}(\norm{\tilde A}_2 + \abs{\tilde c})$ for all $x=(A, c) \in X_1 \times X_2$.
    We may thus identify $X_1$ and $X_2$ with $\R^{d \times d}$ and $\R$ if we weigh the norms by $\abs{\Omega}^{1/2}$.
    We have
    \[
        \iprod{\grad u}{A\grad w}_{L^2(\Omega)} = \int_{\Omega} \tilde A \outerprod \grad u \otimes \grad w\d\xi = \tilde A \outerprod \int_\Omega \grad u \otimes\grad w\d\xi.
    \]
    Thus
    \[
        \xop(u,w) = \begin{pmatrix}
            \int_\Omega \grad u \otimes\grad w\d{\xi}\\
            \int_\Omega uw\d{\xi}
        \end{pmatrix}.
    \]
    According to \cref{lemma:examples:w:assumptions}, the constant
    \[
        C_x =\sup_{(A, c) \in \Dom F}~ \norm{\tilde A - \tilde{\opt A}}_2 + \abs{\tilde c-\tilde{\opt c}}.
    \]
\end{example}

\subsection{Splittings}
\label{sec:examples:splittings}

We now discuss linear system splittings and \cref{assump:convergence:splitting}.
Throughout this subsection we assume that
\begin{equation}
    \label{eq:splitting:setting}
    B(u, w; x)=\dualprod{A_x u + f_x}{w}
    \quad\text{and}\quad
    Lw=\dualprod{b}{w}
\end{equation}
with $A_x \in \linear(U; W^*)$ invertible for $x \in X$, and $f_x, b \in W^*$.
Then for fixed $x \in X$ the weak PDE \eqref{eq:pde-u} and the adjoint $\uop(\freevar, w, x) = - Q'(u)$ reduce to the  linear equations
\[
    A_x u = b - f_x
    \quad\text{and}\quad
    A_x^{*} w = -  Q'(u),
\]
where $A_x^* \in \linear(W; U^*)$ is the dual product adjoint of $A_x$ restricted to $W \hookrightarrow W^{**}$.

\paragraph{The basic splittings}

The next lemma helps to prove \cref{assump:convergence:splitting} subject to a control on the rate of dependence of $A$ on $x$.
In its setting, with $A_x=N_x+M_x$ with $N_x$ “easily” invertible, \cref{step:alg:pde,step:alg:adjoint-pde} of \cref{alg:alg} are given by \eqref{eq:algorithm:splitting-update}.

\begin{theorem}
    \label{thm:splitting:helper}
    In the setting \eqref{eq:splitting:setting}, suppose \cref{assump:convergence:structural} holds and
    \begin{equation}
        \label{eq:convergence:splitting:a-condition}
        \norm{A_x - A_{\tilde x}}_{\linear(U; W^*)} \le L_A \norm{x-\tilde x}_{X}
        \quad\text{and}\quad
        \norm{f_x - f_{\tilde x}}_{W^*} \le L_f \norm{x-\tilde x}_X
        \quad (x, \tilde x \in \Dom F)
    \end{equation}
    for some $L_A \ge 0$.
    Split $A_x=N_x+M_x$ with $N_x$ invertible, and assume there exist $\alpha \in [0, 1)$ and $\gamma_N>0$ such that all
    \begin{equation}
        \label{eq:convergence:splitting:split-condition}
        \|\inv N_x M_x\|_{\linear(U; U)},\norm{N_x^{-1,*}M_x^*}_{\linear(W; W)} \le \alpha
        \quad\text{and}\quad
        \gamma_N \norm{\inv N_x}_{\linear(W^*; U)} \le 1
        \quad (x \in \Dom F).
    \end{equation}
    Also suppose $\grad Q$ is $L_Q$-Lipschitz.
    For any $\gamma_B \in (1, 1/\alpha^2)$, $\lambda \in (0, 1)$, and $\beta>0$, set
    \begin{align*}
        \pi_w & = \left(1+\beta+\frac{\alpha^2\gamma_B}{\lambda(1-\alpha^2\gamma_B)}\right)\frac{\gamma_BL_A^2\norm{\optw}^2_{W}}{\gamma_N^2},
        &
        C_Q & = \left(\frac{1+\beta}{\beta} + \frac{\alpha^2\gamma_B}{(1-\lambda)(1-\alpha^2\gamma_B)}\right)\frac{\gamma_B L_Q^2}{\gamma_N^2},
        \quad\text{and}
        \\
        \pi_u & \omit{\rlap{$\displaystyle= \left(1+\beta+\frac{\alpha^2\gamma_B}{\lambda(1-\alpha^2\gamma_B)}\right)\frac{\gamma_BL_A^2\norm{\optu}^2_{U}}{\gamma_N^2} + \left(\frac{1+\beta}{\beta} + \frac{\alpha^2\gamma_B}{(1-\lambda)(1-\alpha^2\gamma_B)}\right)\frac{\gamma_B L_f^2}{\gamma_N^2}$.}}
    \end{align*}
    Let $\Gamma_k(u, w, x)=\dualprod{M_x u}{w}$ and $\Upsilon_k(u, w, x)=\dualprod{u}{M_x^{*} w}$.
    Then \cref{assump:convergence:splitting} holds for all $k \in \N$ with $\{\nxt v\}_{k=0}^\infty$ generated by \cref{alg:alg} for any $v^0 \in U \times W \times X \times Y$.
\end{theorem}

\begin{proof}
    \Cref{assump:convergence:splitting}\,\cref{item:convergence:splitting:linear} holds by construction, and \cref{item:convergence:splitting:existence} by the assumed invertibility of $N_x$ for $x \in \Dom F$.
    We only consider the second inequality of  \cref{item:convergence:splitting:principal} for $\Upsilon$, the proof of the first inequality for $\Gamma$ being analogous with $-Q'(u)$ replaced by $b-f_x$.
    We thus need to prove
    \begin{equation}
        \label{eq:splitting:helper:target}
        \norm{\thisw-\optw}^2_{W}
        \ge
        \gamma_B \norm{\nxt w-\optw}^2_{W}
        - C_Q\norm{\nxt u - \opt u}^2_{U}
        - \pi_B \norm{\this x - \optx}^2_{X}.
    \end{equation}
    Using \eqref{eq:algorithm:splitting-update} with $A_{\optx}^{*} \opt w =  -Q'(\optu)$ and $A_{\thisx}^{*} \opt w = N_{\thisx}^{*}\optw + M_{\thisx}^{*}\optw$, we expand
    \[
        \begin{aligned}[t]
        \nxt w - \opt w
        &
        = N_{\thisx}^{-1,*}(-  Q'(\nxt u) - M_{\thisx}^{*} \thisw) - \opt w
        \\
        &
        = N_{\thisx}^{-1,*}[ Q'(\opt u) -  Q'(\nxt u)] + N_{\thisx}^{-1,*}(A_{\optx}^{*}-A_{\thisx}^{*})\opt w -  N_{\thisx}^{-1,*}M_{\thisx}^* (\thisw - \optw).
        \end{aligned}
    \]
    Expanding $\norm{\nxt w-\optw}_{W}^2$ and applying the triangle inequality, and Young's inequality thrice, yields
    \[
        \begin{aligned}[t]
        \norm{\nxt w-\optw}^2_{W}
        &
        \le
        \left(1+\frac{\alpha^2\gamma_B}{\lambda(1-\alpha^2\gamma_B)}+\beta\right)\norm{N_{\thisx}^{-1,*}(A_{\optx}^{*}-A_{\thisx}^{*})\opt w}^2_{W}
        + \frac{1}{\alpha^2\gamma_B}\norm{N_{\thisx}^{-1,*}M_{\thisx}^{*} (\thisw - \optw)}^2_{W}
        \\
        \MoveEqLeft[-1]
        + \left(\frac{1 + \beta}{\beta} + \frac{\alpha^2\gamma_B}{(1-\lambda)(1-\alpha^2\gamma_B)}\right)\norm{N_{\thisx}^{-1,*}[Q'(\nxt u) - Q'(\optu)]}^2_{W}.
        \end{aligned}
    \]
    Note that the first part of \eqref{eq:convergence:splitting:a-condition} and the second part \eqref{eq:convergence:splitting:split-condition} hold also for the adjoints $A_x^*$ and $N_x^*$ in the corresponding spaces. Therefore, we establish
    $\norm{N_{\thisx}^{-1,*}(A_{\optx}^{*}-A_{\thisx}^{*})\opt w}^2_{W} \le \gamma_N^{-2}L_A^2\norm{\opt w}^2_{W}\norm{\optx-\thisx}^2_{X}$, $\norm{N_{\thisx}^{-1,*}[ Q'(\nxt u) -  Q'(\optu)]}^2_{W} \le \gamma_N^{-2} L_Q^2 \norm{\nxt u - \opt u}^2_{X}$, and $\norm{N_{\thisx}^{-1,*}M_{\thisx}^{*} (\thisw - \optw)}^2_{W} \le \alpha^2\gamma_B \norm{\thisw - \optw}^2_{W}$.
    Taking $\pi_w$ and $C_Q$ as stated, we therefore obtain \eqref{eq:splitting:helper:target}.
\end{proof}

For our first, infinite-dimensional example of the satisfaction of the conditions of \cref{thm:splitting:helper}, and hence of \cref{assump:convergence:splitting}, note that we have in general
\[
    \norm{\inv N_x}_{\linear(W^*; U)}
    =
    \sup_{w^*} \frac{\norm{\inv N_x w^*}_{U}}{\norm{w^*}_{W^*}}
    =
    \sup_{u} \frac{\norm{u}_{U}}{\norm{N_x u}_{W^*}}
    =
    \sup_{u} \inf_{w} \frac{\norm{u}_{U}\norm{w}_{w}}{\dualprod{N_x u}{w}}
\]
and
\[
    \norm{A_x - A_{\tilde x}}_{\linear(U; W^*)}
    =
    \sup_u \frac{\norm{[A_x-A_{\tilde x}]u}_{W^*}}{\norm{u}_U}
    =
    \sup_{u,w} \frac{\dualprod{[A_x-A_{\tilde x}]u}{w}}{\norm{u}_U\norm{w}_W}.
\]

\begin{example}[No splitting of a weighted Laplacian in $H^1$]
    \label{ex:splitting:laplacian}
    Let $U=W=H_0^1(\Omega)$, $X=\R$, and $N_x=A_x=x\grad^*\grad \in \linear(H_0^1(\Omega); H^{-1}(\Omega))$ be the Laplacian weighted by $x \in (0, \infty)$. Then
    \[
         \norm{\inv N_x}_{\linear(W^*; U)}
         =
         \sup_{u} \inf_{w} \frac{\norm{u}_{H^1(\Omega)}^2}{x\iprod{\grad u}{\grad w}_{L^2(\Omega)}}
         \le
         \sup_{u}  \frac{\norm{u}_{H^1(\Omega)}^2}{x\norm{\grad u}_{L^2(\Omega)}^2}.
    \]
    Therefore, assuming $\inf \Dom F > 0$, we can in \eqref{eq:convergence:splitting:split-condition} take $\gamma_N=\inf_{x \in \Dom F} x\lambda$ for $\lambda$ the infimum of the spectrum of the Laplacian as a bounded self-adjoint operator in $H_0^1(\Omega)$; see, e.g., \cite[Theorem 9.2-1]{kreyszig1991introductory}.
    Clearly also $\alpha=0$ due to $M_x=0$. For \eqref{eq:convergence:splitting:a-condition}, we get
    \[
        \norm{A_x - A_{\tilde x}}_{\linear(U; W^*)}
        =
        \sup_{u,w} (x-\tilde x)\frac{\iprod{\grad u}{\grad w}_{L^2(\Omega)}}{\norm{u}_{H^1(\Omega)}\norm{w}_{H^1(\Omega)}}
        =
        \sup_{u} (x-\tilde x)\frac{\norm{\grad u}_{L^2(\Omega)}^2}{\norm{u}_{H^1(\Omega)}}.
    \]
    Thus we can take $L_A$ as the supremum of the spectrum of the Laplacian as a bounded self-adjoint operator in $H_0^1(\Omega)$.
    
\end{example}

In the following examples, we take $U=W=\R^n$ with the standard Euclidean norm.
Then \eqref{eq:convergence:splitting:split-condition} can be rewritten as the spectral radius bound and positivity condition
\[
    \rho(\inv N_x M_x), \rho(\invstar N_x M_x^*)  \le \alpha
    \quad\text{and}\quad
    N_x^*N_x \ge \gamma_N^2.
\]
The first example also works in general spaces, as seen in a special case in \cref{ex:splitting:laplacian}, but $\gamma_N$ and $L_A$ depend on the norms chosen.
\Cref{thm:splitting:helper} now shows that \cref{assump:convergence:splitting} holds.

\begin{example}[No splitting]
    \label{ex:splitting:no-splitting}
    If $N_x=A_x \in \R^{n \times n}$, \eqref{eq:convergence:splitting:split-condition} holds with $\alpha=0$ and $\gamma_N$ the minimal eigenvalue of $A_x$, assumed symmetric positive definite.
    \Cref{thm:splitting:helper} now shows that \cref{assump:convergence:splitting} holds, where for any $\gamma_B > 1$ and $\beta>0$, we can take
    $
        \pi_w = (1+\beta)\gamma_B \gamma_N^{-2}L_A^2\norm{\optw}^2,
    $
    $
        C_Q = (1+\inv\beta)\gamma_B \gamma_N^{-2} L_Q^2,
    $
    and $\pi_u = \gamma_B \gamma_N^{-2}[(1+\beta)L_A^2\norm{\optu}^2 + (1+\inv\beta)L_f^2]$.
\end{example}

\begin{example}[Jacobi splitting]
    \label{ex:splitting:jacobi}
    If $N_x$ is the diagonal of $A_x \in \R^{n \times n}$, we obtain Jacobi splitting.
    The first part of \eqref{eq:convergence:splitting:split-condition} reduces to strict diagonal dominance, see \cite[§10.1]{golub1996matrix}. The second part always holds and $N_x$ is invertible when the diagonal of $A_x$ has only positive entries. Then $\gamma_N$ is the minimum of the diagonal values.
    \Cref{thm:splitting:helper} now shows that \cref{assump:convergence:splitting} holds.
\end{example}

\begin{example}[Gauss--Seidel splitting]
    \label{ex:splitting:gauss-seidel}
    If $N_x$ is the lower triangle and diagonal of $A_x \in \R^{n \times n}$, we obtain Gauss--Seidel splitting.
    The first part of \eqref{eq:convergence:splitting:split-condition} holds for some $\alpha \in [0, 1)$ when $A_x$ is symmetric and positive definite; compare \cite[proof of Theorem 10.1.2]{golub1996matrix}.
    The second part holds for some $\gamma_N$ when $N_x$ is invertible.
    \Cref{thm:splitting:helper} now shows that \cref{assump:convergence:splitting} holds.
\end{example}

\begin{example}[Successive over-relaxation]
    \label{ex:splitting:sor}
    Based on any one of \cref{ex:splitting:jacobi,ex:splitting:gauss-seidel,ex:splitting:no-splitting}, take $\tilde N_x = (1+r)N_x$ and $\tilde M_x = M_x - rN_x$ for some $r>0$.
    Then, for small enough $\gamma_B$, all $\pi_u, \pi_w, C_Q \downto 0$ as $r \upto \infty$.

    Indeed, $\inv{\tilde N_x}\tilde M_x z = \tilde\lambda z$ if and only if
    $M_x z = ((1+r)\tilde\lambda +r) N_x z $,
    which gives the eigenvalues $\tilde\lambda$ of $\inv{\tilde N_x}\tilde M_x$ as $\tilde\lambda=(\lambda-r)/(1+r)$ for $\lambda$ an eigenvalue of $\inv N_x M_x$.
    So, for large $r$, we can in \eqref{eq:convergence:splitting:split-condition} take $\alpha=(r+\rho)/(1+r)$ and  $\gamma_{\tilde N} = \gamma_N (1+r)$, where $\rho \defeq \rho(\inv{N_x}M_x) < 1$.
    Now, for every large enough $r>0$, for $\gamma_B=(1+\alpha^{-2})/2 > 1$, we have
    \[
        \begin{aligned}
        \frac{\alpha^2}{\gamma_{\tilde N_x}^2(1-\alpha^2\gamma_B)}
        &
        = \frac{2\alpha^2}{\gamma_{\tilde N_x}^2(1-\alpha^2)}
        = \frac{2(1+r)^2\alpha^2}{(1+r)^2\gamma_{N_x}^2((1+r)^2-(1+r)^2\alpha^2)}
        \\
        &
        = \frac{2(r+\rho)^2}{(1+r)^2\gamma_{N_x}^2((1+r)^2-(r+\rho)^2)}
        = \frac{2(r+\rho)^2}{(1+r)^2\gamma_{N_x}^2(1 -\rho^2 + 2(1-\rho)r)}.
        \end{aligned}
    \]
    Since $0 \le \rho<1$, the right hand side tends to zero as $r \upto \infty$.
    Since also $1/\gamma_N^2 \downto 0$, and $\gamma_B>1$, \cref{thm:splitting:helper} now shows that \cref{assump:convergence:splitting} holds with $\pi_u, \pi_w, C_Q \downto 0$ as $r \upto \infty$.
\end{example}

\paragraph{Quasi-conjugate gradients}

With $f_x=0$ for simplicity, motivated by the conjugate gradient method for solving $A_x u = b$, see, e.g., \cite{golub1996matrix}, we propose to perform on \cref{step:alg:pde} of \cref{alg:alg}, and analogously \cref{step:alg:adjoint-pde} the quasi-conjugate gradient update
\begin{equation}
    \label{eq:splitting:quasi-cg}
    \left\{
    \begin{aligned}
        \this r & \defeq b - A_{\this x} \this u, \\
        z^{k+1} & \defeq -\iprod{p^k}{A_{\this x} r^k}/\norm{p^k}_{A_{\this x}}^2, \\
        p^{k+1} & \defeq r^k + z^{k+1} p^k, \\
        t^{k+1} & \defeq \iprod{p^{k+1}}{r^k}/\norm{p^{k+1}}_{A_{\thisx}}^2, \\
        u^{k+1} & \defeq u^k + t^{k+1}p^{k+1}.
    \end{aligned}
    \right.
\end{equation}
For standard conjugate gradients $A_{\thisx} \equiv A$ permits a recursive residual update optimization that we are unable to perform.
We have $\iprod{A_{\thisx}\nxt p}{\this p}=0$ for all $k$, although no “$A$-conjugacy” relationship necessarily exists between $\nxt p$ and $p^j$ for $j < k$.

The next lemma molds the updates \eqref{eq:splitting:quasi-cg} into our overall framework.

\begin{lemma}
    \label{lemma:splitting:cg-gamma}
    The update \eqref{eq:splitting:quasi-cg} corresponds to \cref{step:alg:pde} of \cref{alg:alg} with
    \begin{equation}
        \label{eq:splitting:gamma-cg}
        \Gamma_k(u, \freevar, x)
        = \left[\Id -  \norm{p^{k+1}}_{A_x}^{-2} A_x \left(p^{k+1} \otimes p^{k+1}\right)\right](A_x u^k-b)
        \quad (u \in U).
    \end{equation}
    for $\nxt p = \this r_x + \nxt z_x \this p$ for $\nxt z_x = -\iprod{\this p}{A_x\this r_x}/\norm{\this p}_{A_x}^2$ and $\this r_x \defeq A_x \this u - b$.
\end{lemma}

\begin{proof}
    Indeed, expanding $t^{k+1}$, the $u$-update of \eqref{eq:splitting:quasi-cg} may be rewritten as
    \[
        u^{k+1} - u^k = \norm{p^{k+1}}_{A_{\thisx}}^{-2}(p^{k+1} \otimes p^{k+1}) r^k.
    \]
    Applying the invertible matrix $A_{\thisx}$ and expanding $r^k$, this is
    \[
        A_{\thisx}(u^{k+1}-u^k) =  - \norm{p^{k+1}}_{A_{\thisx}}^{-2} A_{\thisx}(p^{k+1} \otimes p^{k+1})(A_{\thisx} u^k-b),
    \]
    and, adding $A_{\thisx}\thisu-b$ on both sides, further
    \[
        A_{\thisx} u^{k+1}-b = [\Id - \norm{p^{k+1}}_{_{\thisx}}^{-2} A_{\thisx} (p^{k+1} \otimes p^{k+1})](A_{\thisx} u^k-b).
    \]
    Since $B(\nxt u, \freevar; \thisx) = \iprod{A_{\thisx} \nextu}{\freevar}$, and $L(\freevar)=\iprod{b}{\freevar}$, the claim follows.
\end{proof}

Unless $A_x$ is independent of $x$, a simple approach as in \cref{thm:splitting:helper} can only verify \cref{assump:convergence:splitting} with $\gamma_B<1$.
We hence leave the verification of convergence of \cref{alg:alg} with quasi-conjugate gradient updates to future research.

\subsection{Discussion}
\label{sec:examples:discussion}

Before we embark on numerical experiments, it is time to make a few unifying observations about the disparate results above, with regard to the main conditions \cref{eq:convergence:balance-accel:2,eq:convergence:balance-linear:2} of the convergence \cref{thm:convergence:accel,thm:convergence:linear}, and their connection to the fundamentally \emph{discrete} viewpoint of \cref{ex:splitting:jacobi,ex:splitting:gauss-seidel}.
As we have already noted in \cref{rem:linear-pde-condition},
\begin{enumerate}[label=(\roman*)]
    \item The main conditions \cref{eq:convergence:balance-accel:2,eq:convergence:balance-linear:2}  are easily satisfied for linear PDEs, i.e., when $B_x$ does not depend on $u$. In \cref{sec:examples:splittings}, this corresponds to $A_x=A$ (while $f_x$ may still depend on $x$).
    The only condition given in \cref{rem:linear-pde-condition} was that $\pi_w=0$, which is satisfied in \cref{ex:splitting:jacobi,ex:splitting:gauss-seidel,ex:splitting:no-splitting} due to $L_A=0$.
\end{enumerate}
For linear PDEs, $\SupNorm(\optw)=0$. Together with $\pi_w=0$, this causes also $\SupNorm(\optu)$ and $\pi_u$ to disappear from the convergence conditions.
All of these quantities \emph{might} depend on the discretisation.

As we have seen in \cref{sec:examples:W}, $\SupNorm(\optu)$ and $\SupNorm(\optw)$ require the use of $\infty$-norm bounds on the solutions, even when the underlying space is $H^k$. Such bounds may not always hold in infinite dimensions (however, see \cref{rem:examples:w:bounds}), although they do always hold in finite-dimensional subspaces.
In our numerical experiments, we have, however, not observed any grid dependency of $\SupNorm(\optu)$ and $\SupNorm(\optw)$ (calculated a posteriori, after a very large number of iterations).

On a more negative note, with $U=W=\R^{n(h, d)}$ equipped with the standard Euclidean norm, consider $A_x=-x\Delta_h$ for a scalar $x$ with $\Delta_h$ a finite differences discretisation of the Laplacian on a $d$-dimensional square grid of cell width $h$ and $n(h, d)$ nodes.
Then, for both Jacobi and Gauss–Seidel splitting, as well as the trivial splitting (gradient descent) $N_x \propto \Id$, the spectral radius $\rho(\inv N_x M_x) \upto 1$ as $h \downto 0$; see, e.g., \cite[Chapter 4.2.1]{leveque2007fdm}. By simple numerical experiments, $L_A^2/\gamma_N^2$ nevertheless stays roughly constant, so the result is that $\pi_u, \pi_w \upto \infty$ as $h \downto 0$.
For “no splitting”, i.e., $N_x=A_x$, instead $L_A^2/\gamma_N^2 \upto \infty$ due to the worsening condition number of $\Delta_h$. This latter negative result is, however, dependent on taking $U=W=\R^{n(h, d)}$ with the standard Euclidean norm: in \cref{ex:splitting:laplacian} we showed that “no splitting” is applicable to the same problem in $H^1$.
It is, therefore, an interesting question for future research, whether a change of norms would remove the grid dependency of Jacobi and Gauss–Seidel. Our guess is that it would not.

The above indicates that, for nonlinear PDEs, whether our methods even convergence, can depend on the level of discretisation.
Nevertheless, to help comes the successive over-relaxation of \cref{ex:splitting:sor}, which shows that
\begin{enumerate}[resume*]
    \item By letting the over-relaxation parameter $r \upto \infty$, we get $\pi_u, \pi_w, C_Q \downto 0$, and therefore may be able to obtain convergence (with a comparable iteration count) for any magnitude of $\SupNorm(\optu)$, $\SupNorm(\optw)$.
\end{enumerate}

With over-relaxation $\gamma_B \downto 1$ as $r \upto \infty$, so even then, to satisfy \cref{eq:convergence:balance-accel:2,eq:convergence:balance-linear:2}, it is necessary to have very small $C_x$. However,
\begin{enumerate}[resume*]
    \item In \cref{sec:convergence,sec:examples:W}, we have bounded $C_x$ through $\Dom F$, obtaining global convergence when  \cref{eq:convergence:balance-accel:2,eq:convergence:balance-linear:2} hold. With a more refined analysis, it is possible to make $C_x$ arbitrary small by sufficiently good initialisation, i.e., by being content with mere local convergence.
\end{enumerate}
We include a sketch of this analysis in \cref{sec:localisation}.

Finally, although convergence rates ($O(1/N^2)$ or linear) are unaffected by the discretisation level, constant factors of convergence depend on $Z_k\tilde M_k$ through the bound \eqref{eq:convergence:summed-estimate}. This operator, written out in \eqref{eq:convergence:zktildemk}, depends on the constants $\pi_u$ and $\pi_w$.
They inversely scale the magnitude of the testing parameters $\lambda_k$ and $\theta_k$ as chosen in \eqref{eq:convergence:balance:lambdathetadef}.
By \eqref{eq:convergence:balance:lambdathetapi}, the term $\varphi_k+\lambda_k\pi_u + \theta_k\pi_w$ in \eqref{eq:convergence:zktildemk} is, however, independent of $\pi_u$ and $\pi_w$.
Smaller $\pi_u$ and $\pi_w$ are, hence, better for the convergence of $u$ and $w$ (by weighing down the $x$ and $y$ initialisation errors on the right hand side of \eqref{eq:convergence:summed-estimate}), and higher $\pi_u$ and $\pi_w$ are better for the convergence of $x$ and $y$ (by weighing down $u$ and $w$ initialisation errors).
Even for linear PDEs, therefore
\begin{enumerate}[resume*]
    \item Convergence speed may depend on the level of discretisation through the $x$-sensitivity factors $\pi_u$ and $\pi_w$ of the splitting method for the PDE.
\end{enumerate}
This is to be expected: the linear system solvers that \cref{sec:examples:splittings} is based on, are \emph{fundamentally discrete}, and their convergence depends on the eigenvalues of $\inv N_x M_x$ and $N_x$.
In “standard” optimisation methods, the dimensionally-dependent linear system solver is taken as a black box, and its computational cost is hidden from the estimates for the optimisation method.
The estimates for our method, by contrast, include the solver.

\section{Numerical results}
\label{sec:numerics}

We now illustrate the numerical performance of \cref{alg:alg}.
We first describe our experimental setup, and then discuss the results.

\subsection{Experimental setup}

The PDEs in our numerical experiments take one of the forms of \cref{sec:examples:W} on the domain $\Omega = [0,1]\times [0,1]$ with nonhomogeneous Dirichlet boundary conditions.
We discretize the domain as a regular grid and the PDEs by backward differences.
We use both a coarse and a fine grid.

The function $G$ and the PDE vary by experiment, but in each one we take the regularization term for the control parameter $x$ and the data fitting term as
\begin{gather} \label{eq:experiments:F-and-Q}
    F(x) \defeq \frac{\alpha}{2}\norm{x}_{L^2(\Omega; \R^{d \times d}) \times L^2(\Omega)}^2 + \delta_{[\lambda,\lambda^{-1}]}(x)
    \quad\text{and}\quad
    Q(u) \defeq \widehat{\beta} \sum_{i=1}^m\norm{u_i - z_i}_{L^2(\Omega)}^2
\end{gather}
for some $\alpha, \beta, \lambda > 0 $ as well as $\widehat{\beta} \defeq \beta/(2\norm{\bar z}_{L^2(\Omega)}^2)$ where $\bar{z} = \frac{1}{m}\sum_{i=1}^m z_i$ is the average of the measurement data $z_i$.
The norms here are in function spaces, but in the numerical experiments the variables are, of course, taken to be in a finite-dimensional (finite element) subspace.

The variables $u_i$ correspond to multiple copies of the same PDE with different boundary conditions $u_i = f_i$ on $\partial \Omega$, ($i = 1,\dots,m$), for the same control $x$.
Parametrizing $\partial\Omega$ by $\rho : (0,1) \to \partial\Omega$, we take as boundary data
\begin{gather}
    \label{eq:experiments:f-bdry}
    f_{2j-1}(\rho(t)) = \cos(2\pi jt) \quad \text{and} \quad f_{2j}(\rho(t)) = \sin(2\pi jt),
    \quad
    (j=1,\ldots,m/2).
\end{gather}

To produce the synthetic measurement $z_i$, we solve for $\hat u_i$ the PDE corresponding to the experiment with the ground truth control parameter $\hat x = (\hat{A},\hat{c}) $ and boundary data $f_i$.
To this we add Gaussian noise of standard deviation $0.01\norm{\hat{u}_i}_{L^2(\Omega)}$ to get $z_i$.

We next describe the PDEs for each of our experiments.

\begin{experiment}[Scalar coefficient]
\label{num:example:1}
In our first numerical experiment, we aim to determine the scalar coefficient $ c \in \R $ for the PDEs
\begin{align}
    \label{eq:numerics:pde1}
    \left\{\begin{aligned}
        -\Delta u_i + cu_i &= 0 && \text{in $ \Omega $}, \\
        u_i &= f_i && \text{on $ \partial\Omega $},
    \end{aligned}\right.
\end{align}
where $i=1,\ldots,m$.
For this problem we choose $ G(Kx) = 0 $. Thus the objective is
\begin{equation}
    \label{eq:numerics:problem1}
    \min_{u,c} J(x) \defeq \frac{\alpha}{2}\norm{c\mathbf{1}}_{L^2(\Omega)}^2 + \delta_{[\lambda, \lambda^{-1}]}(c) + \widehat{\beta}\sum_{i=1}^m\norm{u_i - z_i}_{L^2(\Omega)}^2
    \quad\text{subject to \eqref{eq:numerics:pde1}}.
\end{equation}
Our parameter choices can be found in \cref{tbl:parameters}.

With $u = (u_1,\ldots, u_m) \in U^m \subset{} H^1(\Omega)^m$ and $w=(w_{1,\Omega}, \ldots, w_{m,\Omega},w_{1,\partial},\ldots,w_{m,\partial}) \in W^m \subset H_0^1(\Omega)^m \times H^{1/2}(\partial \Omega)^m$, for the weak formulation of \eqref{eq:numerics:pde1} we take
\begin{gather}
    \nonumber
    B(u, w; c) = \sum_{i=1}^m\left(
        \iprod{\grad u_i}{\grad w_{i,\Omega}}_{L^2(\Omega)}
        + c \iprod{u_i}{w_{i,\Omega}}_{L^2(\Omega)}
        + \iprod{\trace_{\partial \Omega} u_i}{w_{i,\partial}}_{L^2(\partial\Omega)}
    \right)
\shortintertext{and}
    \label{eq:numerics:pde1:rhs}
    Lw = \sum_{i=1}^m\iprod{f_i}{w_{i,\partial}}_{L^2(\partial\Omega)}.
\end{gather}
Then $ \xop(u,w) = \sum_{i=1}^m \iprod{u_i}{w_{i,\Omega}}_{L^2(\Omega)} $ following \cref{ex:W:scalar}.

For data generation we take $\hat c = 1.0 $.
Since we are dealing with an ill-posed inverse problem, an optimal control parameter $\opt c$ for \eqref{eq:numerics:problem1} does not in general equal $\hat c$.
Therefore, to compare algorithm progress, we take as surrogate for the unknown $\opt c$ the iterate $ \tilde{c}_A \defeq c^{50,000} $ on the coarse grid and $\tilde c_B \defeq c^{500,000} $ on the fine grid, each computed using \cref{alg:alg} without splitting.

The next theorem verifies the basic structural conditions of the convergence \cref{thm:convergence:accel,thm:convergence:linear}. The splitting conditions contained \cref{assump:convergence:splitting} are ensured through \cref{ex:splitting:jacobi} (Jacobi), \ref{ex:splitting:gauss-seidel} (Gauss--Seidel), or \ref{ex:splitting:no-splitting} (no splitting).

\begin{theorem}
	\label{thm:numerics:pde1}
    Let $ X = \R $; $U$ a finite-dimensional subspace of $H^1(\Omega)$; and $ W $ a finite-dimensional subspace of $ H_0^1(\Omega) \times H^{1/2}(\partial\Omega) $.
    Let $F$ and $Q$ be given by \eqref{eq:experiments:F-and-Q} along with the PDE \eqref{eq:numerics:pde1} and the boundary conditions $ f_i $ defined as in \eqref{eq:experiments:f-bdry}. Take $G=0$. Then \cref{assump:convergence:structural} holds.
\end{theorem}
\begin{proof}
	The chosen $F$, $Q$ and either $G$ satisfy \cref{assump:convergence:structural}\ref{item:convergence:structural:convexity}. The boundary conditions $ f_i \in H^{1/2}(\partial\Omega)$ along with the constraint $x \in [\lambda,\lambda^{-1}]$ ensure the condition \cref{lemma:examples:w:assumptions}\ref{item:examples:w:assumptions:subspace}. In the discretized setting, also \ref{item:examples:w:assumptions:bound-sup} and \ref{item:examples:w:assumptions:bound} also hold.
    In conclusion, \cref{lemma:examples:w:assumptions} verifies \cref{assump:convergence:structural}.
\end{proof}

\begin{remark}
\label{rem:numerics:pde1}
It remains to verify \eqref{eq:convergence:balance-accel} or \eqref{eq:convergence:balance-linear}, depending on the convergence theorem used.
The condition \eqref{eq:convergence:balance-accel:0} is readily verified by appropriate choice of the primal and dual step length parameters $\tau_0,\sigma_0>0$.
We also take $\tilde\gamma_F=0$ (slightly violating the assumptions), so that $\omega_k \equiv 1$, and $\tau_k \equiv \tau_0$ and $\sigma_k \equiv \sigma_0$.
The condition \eqref{eq:convergence:balance-accel:2} (and likewise  \eqref{eq:convergence:balance-linear:2} for linear convergence) is very difficult to verify \emph{a priori} for nonlinear PDEs, as it depends on the knowledge of a solution to the optimisation problem through $\SupNorm(\optu)$ and $\SupNorm(\optw)$.
This is akin to the difficulty of verifying (a priori) a positive Hessian at a solution for standard nonconvex optimisation methods.
Hence we do not attempt to verify \eqref{eq:convergence:balance-accel:2}.
\end{remark}
\end{experiment}

\begin{experiment}[Diffusion + scalar coefficient]
\label{num:example:2}
In this experiment we aim to determine the coefficient function $ a : \Omega \to \R $ and scalar $ c \in \R $ for the group of PDEs
\begin{align}
    \label{eq:numerics:pde2}
    \left\{\begin{aligned}
        -\grad\cdot (a\grad u_i) + cu_i &= 0 && \text{in $ \Omega $}, \\
        u_i &= f_i && \text{on $ \partial\Omega $},
    \end{aligned}\right.
\end{align}
where $i=1,\ldots,m$. The optimization problem then is
\begin{gather}
	\label{eq:numeric:example2}
    \min_{x = (a,c)} J(x) = \delta_{[\lambda, \lambda^{-1}]}(x) + \widehat{\beta}\sum_{i=1}^m\norm{u_i - z_i}_{L^2(\Omega)}^2 + \gamma\norm{\grad a}_1
    \quad\text{subject to \eqref{eq:numerics:pde2}}.
\end{gather}
Note that, although we take the total variation of $a$, which is natural in the space of functions of bounded variation, we consider $a$ to lie in (as per \cref{ex:algorithm:b-idea} a finite-dimensional subspace of) $L^2(\Omega)$. Thus the total variation term has value $+\infty$ in $L^2(\Omega) \setminus \operatorname{BV}(\Omega)$. Nevertheless, the term is weakly lower semicontinuous even in $L^2$ due to Poincaré's inequalities (for example, \cite[Theorem 3.44]{ambrosio2000fbv}), so the problem is well-defined. Subdifferentiation in $L^2(\Omega)$ is a slightly more delicate issue, but not a problem for optimality conditions of problems of the type \eqref{eq:numeric:example2}, as discussed in \cite[Remark 4.7]{tuomov-regtheory}. Moreover, as said, in practise we work in a finite-dimensional subspace that corresponds to the backward differences discretisation of the gradient in the total variation term. The convergence of discretisations is discussed in \cite{ckp1999regularization}.

For the weak formulation of \eqref{eq:numerics:pde2} with $w=(w_{1,\Omega}, \ldots, w_{m,\Omega},w_{1,\partial},\ldots,w_{m,\partial}) \in W^m \subset H_0^1(\Omega)^m \times H^{1/2}(\partial \Omega)^m$, $u = (u_1,\ldots, u_m) \in U^m\subset{}H^1(\Omega)^m$, and $x=(a,c) \in X \subset L^2(\Omega) \times \R$, we take $L$ as in \eqref{eq:numerics:pde1:rhs} and
\[
    B(u, w; x) = \sum_{i=1}^m\left(
        \iprod{\grad u_i}{a\grad w_{i,\Omega}}_{L^2(\Omega)}
        + c \iprod{u_i}{w_{i,\Omega}}_{L^2(\Omega)}
        + \iprod{\trace_{\partial \Omega} u_i}{w_{i,\partial}}_{L^2(\partial\Omega)}
    \right).
\]
Then $ \xop(u,w) = (\xop^1(w,u),\xop^2(w,u))$ takes on a mixed form with $ \xop^1(w,u) = \sum_{i=1}^m \nabla u_i\cdot\nabla w_{i,\Omega} $ from \cref{ex:W:aI} and $\xop^2(w,u) = \sum_{i=1}^m \iprod{u_i}{w_{i,\Omega}}_{L^2(\Omega)} $ from \cref{ex:W:scalar}.

For data generation we take $ \hat c = 1.0 $ and $ \hat a $ as the phantom in \cref{fig:case4a:experiments:recon}.
Similarly to \cref{num:example:1} we compare the progress towards $ \tilde a \defeq a^{1,000,000} $ and $ \tilde c \defeq c^{1,000,000} $ computed using \cref{alg:alg} with full matrix inversion.

As above for \cref{num:example:1}, the next theorem verifies the basic structural conditions of the convergence \cref{thm:convergence:accel,thm:convergence:linear}. The proofs is analogous to that \cref{thm:numerics:pde1}. Likewise, the splitting \cref{assump:convergence:splitting} is verified as before through \cref{ex:splitting:jacobi} (Jacobi), \ref{ex:splitting:gauss-seidel} (Gauss--Seidel), or \ref{ex:splitting:no-splitting} (no splitting), while \cref{rem:numerics:pde1} applies for the remaining step length and growth conditions.

\begin{theorem}
	\label{thm:numerics:pde2}
    Let $X$ be a finite-dimensional subspace of $L^2(\Omega) \times \R $, $U$ a finite-dimensional subspace of $H^1(\Omega)$ and $ W $ a finite-dimensional subspace of $ H_0^1(\Omega) \times H^{1/2}(\partial\Omega) $.
    Let $F$ and $Q$ be given by \eqref{eq:experiments:F-and-Q} along with the PDE \eqref{eq:numerics:pde2} with the boundary conditions $ f_i $ defined as in \eqref{eq:experiments:f-bdry} and $G$ be $ \|\freevar\|_1$. Then \cref{assump:convergence:structural} holds.
\end{theorem}
\end{experiment}

\newlength\threefigwidth
\setlength\threefigwidth{0.275\textwidth}
\newlength\twofigwidth
\setlength\twofigwidth{0.435\textwidth}
\newlength\twofigheight
\setlength\twofigheight{\threefigwidth}

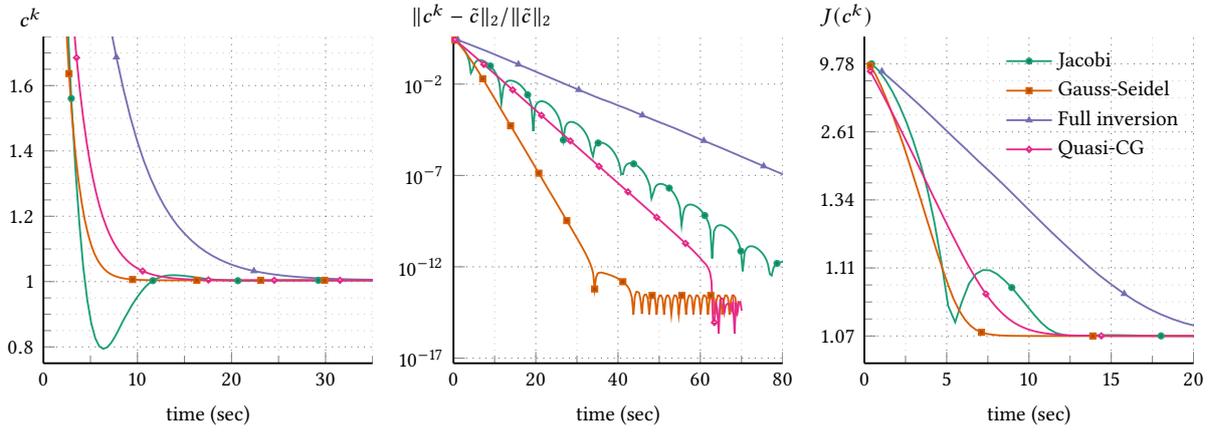
\begin{figure}[t]
        \begin{tikzpicture}
            \begin{axis}[
                width=\threefigwidth,
                height=\threefigwidth,
                convergence style,
                legend ne style,
                xmin=0, xmax=35,
                ymin=0.75, ymax=1.75,
                xlabel={time (sec)},
                ylabel={$c^k$},
            ]
                \addplot table[x=jac_t,y=jac_x]{data/case1a/data.txt};
                \addplot table[x=gs_t,y=gs_x]{data/case1a/data.txt};
                \addplot table[x=none_t,y=none_x]{data/case1a/data.txt};
                \addplot table[x=cg_t,y=cg_x]{data/case1a/data.txt};

            \end{axis}
        \end{tikzpicture}\hfill%
        \begin{tikzpicture}
            \begin{axis}[
                width=\threefigwidth,
                height=\threefigwidth,
                convergence style,
                legend ne style,
                ymode=log,
                xmin=0, xmax=80,
                ymin=0, ymax=4,
                xlabel={time (sec)},
                ylabel={$\norm{c^k-\tilde c}_2/\norm{\tilde c}_2$},
            ]
                \addplot table[x=jac_t,y=jac_relerr]{data/case1a/data.txt};
                \addplot table[x=gs_t,y=gs_relerr]{data/case1a/data.txt};
                \addplot table[x=none_t,y=none_relerr]{data/case1a/data.txt};
                \addplot table[x=cg_t,y=cg_relerr]{data/case1a/data.txt};

            \end{axis}
        \end{tikzpicture}\hfill%
        \begin{tikzpicture}
            \SetMinMax{data/case1a/data.txt}{jac_J}{\jacRes}%
            \UpdMinMax{data/case1a/data.txt}{gs_J}{\gsRes}%
            \UpdMinMax{data/case1a/data.txt}{none_J}{\nonRes}%
            \UpdMinMax{data/case1a/data.txt}{cg_J}{\cgRes}%
            \begin{axis}[
                width=\threefigwidth,
                height=\threefigwidth,
                convergence style,
                legend ne style,
                xmin=0, xmax=20,
                xlabel={time (sec)},
                ylabel={$J(c^k)$},
                legend pos = north east,
                fixedylog = {1000}{},
            ]
                \addplot table[x=jac_t,y=jac_J]{\jacRes};
                \addplot table[x=gs_t,y=gs_J]{\gsRes};
                \addplot table[x=none_t,y=none_J]{\nonRes};
                \addplot table[x=cg_t,y=cg_J]{\cgRes};

                \addlegendentry{Jacobi}
                \addlegendentry{Gauss-Seidel}
                \addlegendentry{Full inversion}
                \addlegendentry{Quasi-CG}

            \end{axis}
        \end{tikzpicture}%
    \caption{Performance of various splittings in the coarse grid \cref{num:example:1}.}
    \label{fig:case1a:experiments}
\end{figure}
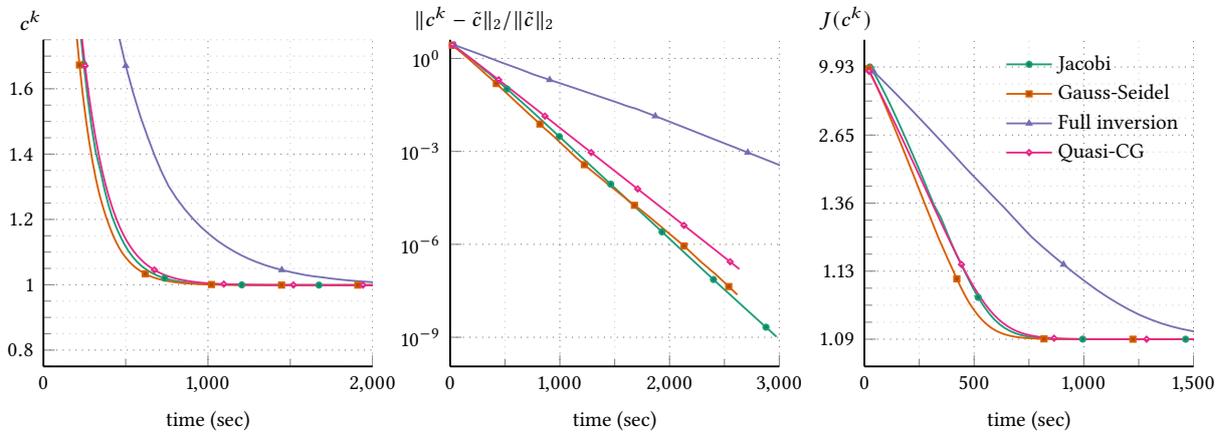
\begin{figure}[t]
        \begin{tikzpicture}
            \begin{axis}[
                width=\threefigwidth,
                height=\threefigwidth,
                convergence style,
                legend ne style,
                xmin=0, xmax=2000,
                ymin=0.75, ymax=1.75,
                xlabel={time (sec)},
                ylabel={$c^k$},
                xtick={0,1000,2000}
            ]
                \addplot table[x=jac_t,y=jac_x]{data/case1b/data.txt};
                \addplot table[x=gs_t,y=gs_x]{data/case1b/data.txt};
                \addplot table[x=none_t,y=none_x]{data/case1b/data.txt};
                \addplot table[x=cg_t,y=cg_x]{data/case1b/data.txt};

            \end{axis}
        \end{tikzpicture}\hfill%
        \begin{tikzpicture}
            \begin{axis}[
                width=\threefigwidth,
                height=\threefigwidth,
                convergence style,
                legend ne style,
                ymode=log,
                xmin=0, xmax=3000,
                ymin=0, ymax=4,
                xlabel={time (sec)},
                ylabel={$\norm{c^k-\tilde c}_2/\norm{\tilde c}_2$},
            ]
                \addplot table[x=jac_t,y=jac_relerr]{data/case1b/data.txt};
                \addplot table[x=gs_t,y=gs_relerr]{data/case1b/data.txt};
                \addplot table[x=none_t,y=none_relerr]{data/case1b/data.txt};
                \addplot table[x=cg_t,y=cg_relerr]{data/case1b/data.txt};

            \end{axis}
        \end{tikzpicture}\hfill%
        \begin{tikzpicture}
            \SetMinMax{data/case1b/data.txt}{jac_J}{\jacRes}%
            \UpdMinMax{data/case1b/data.txt}{gs_J}{\gsRes}%
            \UpdMinMax{data/case1b/data.txt}{none_J}{\nonRes}%
            \UpdMinMax{data/case1b/data.txt}{cg_J}{\cgRes}%
            \begin{axis}[
                width=\threefigwidth,
                height=\threefigwidth,
                convergence style,
                legend ne style,
                xmin=0, xmax=1500,
                xlabel={time (sec)},
                ylabel={$J(c^k)$},
                legend pos = north east,
                xtick={0,500,1000,1500},
                fixedylog = {1000}{},
            ]
                \addplot table[x=jac_t,y=jac_J]{\jacRes};
                \addplot table[x=gs_t,y=gs_J]{\gsRes};
                \addplot table[x=none_t,y=none_J]{\nonRes};
                \addplot table[x=cg_t,y=cg_J]{\cgRes};

                \addlegendentry{Jacobi}
                \addlegendentry{Gauss-Seidel}
                \addlegendentry{Full inversion}
                \addlegendentry{Quasi-CG}

            \end{axis}
        \end{tikzpicture}%
    \caption{Performance of various splittings in fine grid \cref{num:example:1}.}
    \label{fig:case1b:experiments}
\end{figure}

\begin{figure}[t]
        \begin{tikzpicture}
            \begin{axis}[
                width=\twofigwidth,
                height=\twofigheight,
                convergence style,
                legend ne style,
                ymode=log,
                xmin=0, xmax=1000,
                xlabel={time (sec)},
                ylabel={$\norm{(\tilde c/c^k)a^k - \tilde a}_{L^2} / \norm{\tilde a}_{L^2} $}
            ]
                \addplot table[x=jac_t,y=jac_relerr]{data/case4a/data.txt};
                \addplot table[x=gs_t,y=gs_relerr]{data/case4a/data.txt};
                \addplot table[x=none_t,y=none_relerr]{data/case4a/data.txt};
                \addplot table[x=cg_t,y=cg_relerr]{data/case4a/data.txt};

            \end{axis}
        \end{tikzpicture}\hfill%
        \begin{tikzpicture}
            \SetMinMax{data/case4a/data.txt}{jac_J}{\jacRes}%
            \UpdMinMax{data/case4a/data.txt}{gs_J}{\gsRes}%
            \UpdMinMax{data/case4a/data.txt}{none_J}{\nonRes}%
            \UpdMinMax{data/case4a/data.txt}{cg_J}{\cgRes}%
            \begin{axis}[
                width=\twofigwidth,
                height=\twofigheight,
                convergence style,
                legend ne style,
                xmin=0, xmax=1000,
                xlabel={time (sec)},
                ylabel={$J(a^k,c^k)$},
                fixedylog = {1000}{},
            ]
                \addplot table[x=jac_t,y=jac_J]{\jacRes};
                \addplot table[x=gs_t,y=gs_J]{\gsRes};
                \addplot table[x=none_t,y=none_J]{\nonRes};
                \addplot table[x=cg_t,y=cg_J]{\cgRes};

                \addlegendentry{Jacobi}
                \addlegendentry{Gauss-Seidel}
                \addlegendentry{Full inversion}
                \addlegendentry{Quasi-CG}

            \end{axis}
        \end{tikzpicture}
    \caption{Performance of various splittings in the coarse grid \cref{num:example:2}.}
    \label{fig:case4a:experiments}
\end{figure}
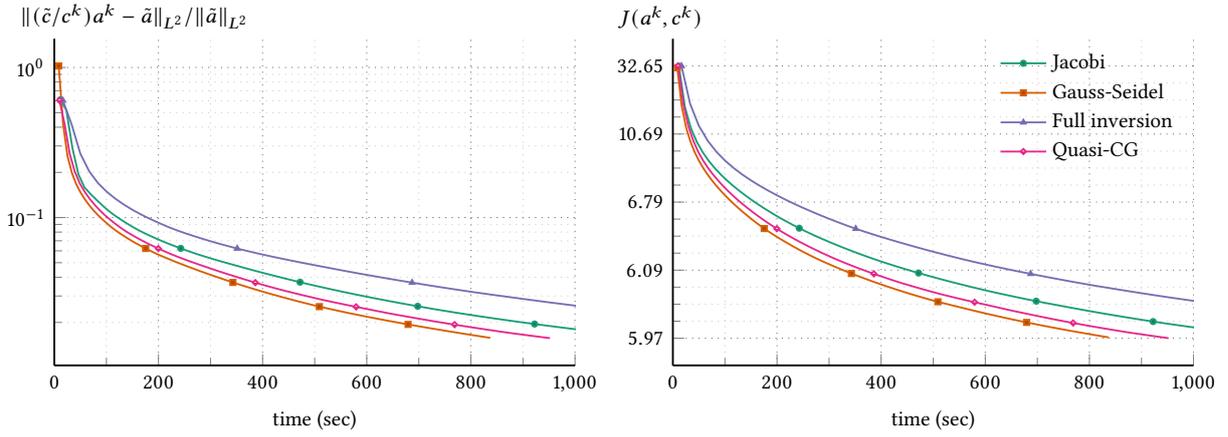
\begin{figure}[t]
        \begin{tikzpicture}
            \begin{axis}[
                width=\twofigwidth,
                height=\twofigheight,
                convergence style,
                legend ne style,
                ymode=log,
                xmin=0, xmax=10000,
                xlabel={time (sec)},
                ylabel={$\norm{(\tilde c/c^k)a^k - \tilde a}_{L^2} / \norm{\tilde a}_{L^2} $}
            ]
                \addplot table[x=jac_t,y=jac_relerr]{data/case4b/data.txt};
                \addplot table[x=gs_t,y=gs_relerr]{data/case4b/data.txt};
                \addplot table[x=none_t,y=none_relerr]{data/case4b/data.txt};
                \addplot table[x=cg_t,y=cg_relerr]{data/case4b/data.txt};

            \end{axis}
        \end{tikzpicture}\hfill%
        \begin{tikzpicture}
            \SetMinMax{data/case4b/data.txt}{jac_J}{\jacRes}%
            \UpdMinMax{data/case4b/data.txt}{gs_J}{\gsRes}%
            \UpdMinMax{data/case4b/data.txt}{none_J}{\nonRes}%
            \UpdMinMax{data/case4b/data.txt}{cg_J}{\cgRes}%
            \begin{axis}[
                width=\twofigwidth,
                height=\twofigheight,
                convergence style,
                legend ne style,
                xmin=0, xmax=10000,
                xlabel={time (sec)},
                ylabel={$J(a^k,c^k)$},
                fixedylog = {1000}{},
            ]
                \addplot table[x=jac_t,y=jac_J]{\jacRes};
                \addplot table[x=gs_t,y=gs_J]{\gsRes};
                \addplot table[x=none_t,y=none_J]{\nonRes};
                \addplot table[x=cg_t,y=cg_J]{\cgRes};

                \addlegendentry{Jacobi}
                \addlegendentry{Gauss-Seidel}
                \addlegendentry{Full inversion}
                \addlegendentry{Quasi-CG}

            \end{axis}
        \end{tikzpicture}
    \caption{Performance of various splittings in the fine grid \cref{num:example:2}.}
    \label{fig:case4b:experiments}
\end{figure}
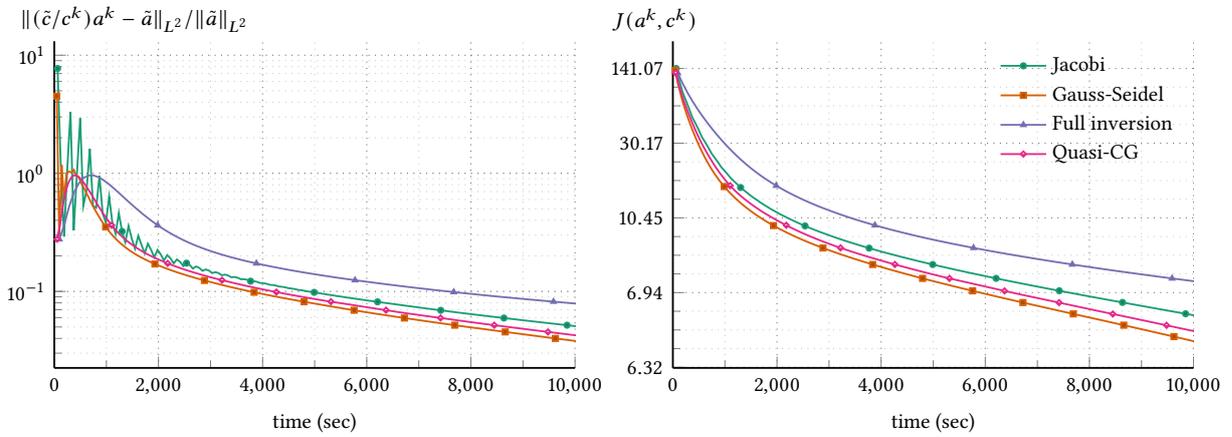
\begin{figure}
        \begin{tikzpicture}
            \begin{axis}[parameter style, title=Jacobi split]
                \addplot graphics [xmin=0,xmax=1,ymin=0,ymax=1] {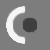};
            \end{axis}
        \end{tikzpicture}%
        \begin{tikzpicture}
            \begin{axis}[
                parameter style,
                yticklabels={,,},
                title=No split]
                \addplot graphics [xmin=0,xmax=1,ymin=0,ymax=1] {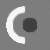};
            \end{axis}
        \end{tikzpicture}%
        \begin{tikzpicture}
            \begin{axis}[
                name = last,
                parameter style,
                yticklabels={,,},
                title=Data generation phantom]
                \addplot graphics [xmin=0,xmax=1,ymin=0,ymax=1] {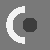};
            \end{axis}
            \begin{axis}[colorbar style, at = {(last.east)}]
                \addplot graphics [xmin=0,xmax=.1575,ymin=0,ymax=2.1] {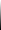};
            \end{axis}
        \end{tikzpicture}
    \caption{Illustrations of the coefficient reconstructions for \cref{num:example:2}A. On the left is the result of the Jacobi split approach, in the middle the full matrix inversion after the same number of iterations. On the right we show the data generation phantom for comparison.}
    \label{fig:case4a:experiments:recon}
\end{figure}
\begin{figure}
        \begin{tikzpicture}
            \begin{axis}[parameter style, title=Jacobi split]
                \addplot graphics [xmin=0,xmax=1,ymin=0,ymax=1] {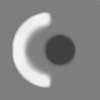};
            \end{axis}
        \end{tikzpicture}%
        \begin{tikzpicture}
            \begin{axis}[
                parameter style,
                yticklabels={,,},
                title=No split]
                \addplot graphics [xmin=0,xmax=1,ymin=0,ymax=1] {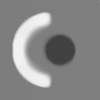};
            \end{axis}
        \end{tikzpicture}%
        \begin{tikzpicture}
            \begin{axis}[
                name = last,
                parameter style,
                yticklabels={,,},
                title=Data generation phantom]
                \addplot graphics [xmin=0,xmax=1,ymin=0,ymax=1] {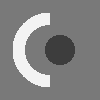};
            \end{axis}
            \begin{axis}[colorbar style, at = {(last.east)}]
                \addplot graphics [xmin=0,xmax=.1575,ymin=0,ymax=2.1] {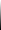};
            \end{axis}
        \end{tikzpicture}
    \caption{Illustrations of the coefficient reconstructions for \cref{num:example:2}B. On the left is the result of the Jacobi split approach, in the middle the full matrix inversion after the same number of iterations. On the right we show the data generation phantom for comparison.}
    \label{fig:case4b:experiments:recon}
\end{figure}

\subsection{Algorithm parametrisation}
\label{sec:numerics:parametrisation}

We apply \cref{alg:alg} with no splitting (full inversion), and with Jacobi and Gauss–Seidel splitting, and quasi conjugate gradients, as discussed in \cref{sec:examples:splittings}.
We fix $ \sigma = 1.0 $, $ \omega = 1.0 $, $ \lambda = 0.1 $, $ \varepsilon = 0.01 $, and $ \beta = 10^2 $ for all experiments. Other parameters, including the grid size, $ \alpha $, $ \gamma_i $, $ \tau $ and $ m $ vary according to experiment with values listed in \cref{tbl:parameters}.

For the initial iterate $ (x^0,u^0,w^0,y^0) $ we make an experiment-specific choice of the control parameter $x^0$. Then we determine $ u^0 $ by solving the PDE, and $ w_0 $ by solving the adjoint PDE. We set $ y^0 = Kx^0 $.
For \cref{num:example:1} we take the initial $ c^0 = 4.0 $ and run the algorithm for 20,000 iterations on the coarse grid and 125,000 on the fine.
For \cref{num:example:2} we take the initial $ a^0 \equiv 1.0$ a constant function, and $ c^0 = 2.0 $. The algorithm is run for 200,000 iterations on the coarse grid, and 500,000 on the fine.

\begin{table}
    \centering
    \caption{Parameter choices for all examples.} \label{tbl:cases:parameters}
    \begin{tabular}{c|c|c|c|c|c|c|c|c|c}
        Grid & $N$ & Grid size & $\alpha$ & $\beta$ & $\gamma$ & $\tau$ & $\sigma$ & $\omega$ & $m$ \\\hline
        Coarse & 51 & 2601 & $1\times 10^{-5}$ & $1\times 10^2$ & $0$ & $ 2.5\times 10^{-2} $ & $ 1 $ & $ 1 $ & $ 6 $ \\
        Fine & 101 & 10201 & $1\times 10^{-5}$ & $1\times 10^2$ & $0$ & $ 2.0\times 10^{-3} $ & $ 1 $ & $ 1 $ & $ 6 $ \\\hline
        Coarse & 51 & 2601 & $ 0 $ & $1\times 10^2$ & $ 10^{-2} $ & $ 2.5\times 10^{-2} $ & $ 1 $ & $ 1 $ & $ 10 $ \\
        Fine & 101 & 10201 & $ 0 $ & $1\times 10^2$ & $ 10^{-2} $ & $ 1\times 10^{-2} $ & $ 1 $ & $ 1 $ & $ 10 $
    \end{tabular}
    \label{tbl:parameters}
\end{table}

We implemented the algorithm in Julia.
The implementation is available on Zenodo \cite{jensen2022codes}.
The experiments were run on a ThinkPad laptop with Intel Core i5-8265U CPU at 1.60GHz $\times 4$ and 15.3 GiB memory.

\subsection{Results}

The results for \cref{num:example:1} with the above algorithm parametrisations are in \cref{fig:case1a:experiments} for the coarse grid and \cref{fig:case1b:experiments} for the fin grid. In the figures we illustrate the evolution of the coefficient $ c^k $ as the algorithm iterates. We also show the evolution of the relative error of the coefficient and the functional value.

The results for \cref{num:example:2} are available in \cref{fig:case4a:experiments,fig:case4a:experiments:recon} for the coarse grid and \cref{fig:case4b:experiments,fig:case4b:experiments:recon}  for the fine grid.
In \cref{fig:case4a:experiments,fig:case4b:experiments} are shown the evolution of the relative error of the coefficient and the functional value.
In \cref{fig:case4a:experiments:recon,fig:case4b:experiments:recon} are the reconstructed coefficients $ a^k $ at the final iterates and for comparison the phantom used for the data generation.

The performance plots have \emph{time} on the $x$-axis rather than the number of iterations, as the main difference between the splittings is expected to be in the computational effort for linear system solution, i.e., \cref{step:alg:pde,step:alg:adjoint-pde} of \cref{alg:alg}.
For fairness, we limited the number of threads used by Julia/OpenBLAS to one.

In all experiments the splittings outperform full matrix inversion: the best splittings require roughly \emph{half of the computational effort} for an iterate of the same quality.
No particular splitting completely dominates another, however, Jacobi appear to be more prone to overstepping and oscillatory patterns.
On the other hand, quasi-CG currently has no convergence theory, and we have observed situations where it does not exhibit convergence while Jacobi and Gauss–Seidel splittings do. Therefore, Gauss–Seidel is our recommended option.

\appendix

\section{Optimality conditions}
\label{sec:oc}

We prove here the necessity of \eqref{eq:algorithm:oc1} for solutions to \eqref{eq:algorithm:minmax}.

\begin{proof}[Proof of \cref{thm:algorithm:oc}]
    We let $T(x, u) \defeq B(u, \freevar; x)$, $T: X \times U \to W^*$.
    Setting
    \[
        A := \{ (x, u) \in X \times U \mid B(u, w; x) = Lw \text{ for all } w \in W\}
        = T^{-1}(L),
    \]
    any solution $(\optu, w, \optx, \opty)$ to \eqref{eq:algorithm:minmax} also solves
    \[
        \min_{x,u} R(x, u) \defeq [R_0 + \delta_A](x, u)
        \quad\text{where}\quad
        R_0(x, u) = F(x) + Q(u) + G(Kx).
    \]
    with $G(K\opt x)=\iprod{K\optx}{\opt y}_Y - G^*(\opt y)$.
    By the Fenchel-Young theorem, the latter is equivalent to the last line of \eqref{eq:algorithm:oc1}.
    Clearly $(\opt x, \opt u) \in A$, or else there is no solution. Therefore also the first line of \eqref{eq:algorithm:oc1} holds.

    It follows from the linearity/affinity and continuity, hence continuous differentiability of $B$ that $T$ is strictly differentiable.
    Since $T'(\optx, \optu)(h_x, h_u) = \linwop(\optu, \freevar; h_x) + \uop(h_u, \freevar; \optx)$, so that
    \[
        \dualprod{T'(\optx, \optu)^*w}{(h_x, h_u)} = \linwop(\optu, w; h_x) + \uop(h_u, w; \optx),
    \]
    the qualification condition \eqref{eq:algorithm:cq1} reads
    \[
        \sup_{\norm{(h_x, h_u)}=1} \norm{T'(\optx, \optu)^*(h_x, h_u)} \ge c \norm{w}
        \quad\text{for all}\quad w \in W.
    \]
    Moreover, as a bounded linear operator, $T'(\optx, \optu)$ is closed, i.e., has closed graph.
    Therefore, by \cite[Theorem 2.20]{brezis2011functional}, $T'(\optx, \optu)$ is surjective.
    With this, \cite[Theorem 1.17]{mordukhovich2006variational} gives
    \[
        \begin{aligned}
        \subdiff_M \delta_A(x, u)
        &
        = T'(\optx, \optu)^*N_{\{L\}}(T(\optx, \optu)),
        \\
        &
        = \{ (h_x, h_u) \mapsto \dualprod{T'(\optx, \optu)(h_x, h_u)}{w} \mid w \in W\}
        \\
        &
        = \{(h_x, h_u) \mapsto \linwop(\optu, w; h_x) + \uop(h_u, w; \optx) \mid w \in W\}.
        \end{aligned}
    \]
    Here we denote by $N_D(x)=\subdiff_M \delta_D(x)$ the limiting normal cone to a set $D$ at $x$.

    Since limiting subdifferentials agree with convex subdifferentials on convex functions, and we have assumed that $\interior \Dom R_0 \ne \emptyset$, we can easily calculate $\subdiff_M R_0$.
    We will then use the sum rule \cite[Theorem 3.36]{mordukhovich2006variational} to estimate $\subdiff_M R$, which requires verifying that $R_0$ is “sequentially normally epicompact” (SNEC), and that the “horizon subdifferentials”, defined for $V: X \to \extR$ as $\subdiff^\infty V(x) \defeq \{ x^* \in X^* \mid (x^*, 0) \in N_{\epi V}(x, V(x))\}$, satisfy
    \begin{equation}
        \label{eq:oc:horizon-cond}
        \subdiff^\infty \delta_A(\optx,\optu) \isect (-\subdiff^\infty R(\optx,\optu)) = \{0\}.
    \end{equation}
    Indeed, convex functions whose domains have a non-empty interior, such as $R_0$, are SNEC by \cite[Proposition 1.25 and discussion after Definition 1.116]{mordukhovich2006variational}.
    Moreover, since $\subdiff^\infty Q(\opt u)=\{0\}$, \eqref{eq:oc:horizon-cond} reduces to
    \[
        \uop(\freevar, w; \optx) = 0
        \implies
        \linwop(\optu, w; \freevar) \isect (- \subdiff^\infty[F + G \circ K](x)) = \{0\}
    \]
    This is guaranteed by the qualification condition \eqref{eq:algorithm:cq2}.
    Now, by the Fermat principle \cite[Proposition 1.114]{mordukhovich2006variational} and the sum rule \cite[Theorem 3.36]{mordukhovich2006variational}, we have
    \[
        0 \in \subdiff_M R(\opt x, \opt u) \subset
        \begin{pmatrix}
            \subdiff F(\opt x) + K^*\subdiff G(K\opt x)
            \\
            \{Q'(\opt u)\}
        \end{pmatrix}
        + \subdiff_M \delta_A(\opt x, \opt u).
    \]
    After appropriate Riesz representations, this inclusion expands as the middle two lines of \eqref{eq:algorithm:oc1}.
\end{proof}

\section{Localization}
\label{sec:localisation}

\Cref{thm:convergence:accel,thm:convergence:linear} are global convergence results, but also depend on the global constant $C_x$ in \cref{assump:convergence:structural}\,\Cref{item:convergence:structural:bound}. To satisfy the conditions of the theorems, $\Dom F$ may need to be small for $C_x$ to be small. We now develop local convergence results that allow replacing $C_x$ by a small initialization-dependent value without restricting $\Dom F$.

We replace  \cref{assump:convergence:structural} with the following:

\begin{assumption}
    \label{assump:local:convergence:structural}
    We assume \cref{assump:convergence:structural} to hold with \cref{item:convergence:structural:bound} replaced by
    \begin{enumerate}[start=4,label=(\roman*$'$)]
        \item\label{item:local:convergence:structural:bound}
        For some $\tilde C_x \ge 0$, for all $(u, w) \in U \times W$ and $x \in \Dom F$ we have the bound
        \[
            \linwop(u, w; x-\opt x)
            \le \tilde C_x \norm{x-\optx}_X\norm{u}_U\norm{w}_W.
        \]
    \end{enumerate}
\end{assumption}

This estimate uses the standard norm in $X$, which is a $2$-norm in the examples of \cref{sec:examples:W,sec:numerics}. However, \cref{sec:examples:W} gives estimates involving an $\infty$-norm for $C_x$. Therefore some finite-dimensionality of the parameters is required to satisfy \cref{assump:local:convergence:structural}\,\cref{item:local:convergence:structural:bound}.
This can take the form of a finite element discretisation of a function parameter $a$, or the parameter being a scalar constant. In the latter case, the examples of \cref{sec:examples:W} readily verify \cref{assump:local:convergence:structural}.

We then modify several previous results accordingly:

\begin{lemma}[Local version of \cref{lem:ineq-Hk-norm-1}]
    \label{lem:local:ineq-Hk-norm-1}
    Let $k \in \N$. Suppose \cref{assump:convergence:testing,assump:local:convergence:structural,assump:convergence:splitting} hold,
    \begin{equation}
        \label{eq:local:ineq-Hk-norm-1:apriori-bound}
        \norm{\nextu-\optu}_U \le \delta_{uw},
        \quad\text{and}\quad
        \norm{\nextw-\optw}_U \le \delta_{uw},
    \end{equation}
    for some $\delta_{uw} > 0$, and for some $\epsilon_u,\epsilon_w,\mu > 0$ that
    \begin{subequations}
        \label{eq:local:convergence:balance0}
        \begin{align}
            \label{eq:local:convergence:balance0:gammaf}
            \gamma_F
            &
            \ge
            \tilde\gamma_F
            + \epsilon_u + \epsilon_w
            + \frac{\lambda_{k+1}\pi_u + \theta_{k+1}\pi_w}{\eta_k}
            \\
            \label{eq:local:convergence:balance0:gammag}
            \gamma_{G^*}
            &
            \ge
            \tilde\gamma_{G^*},
            \\
            \label{eq:local:convergence:balance0:gammab}
            \gamma_B
            &
            \ge
            \frac{\lambda_{k+1}}{\lambda_k}
            + \frac{\theta_k}{\lambda_k}C_Q
            + \frac{\eta_k\SupNorm(\optw)}{2\epsilon_w\lambda_k}
            + \frac{\tilde C_x^2 \delta_{uw}^2 \mu \eta_k}{4\epsilon_u\lambda_k},
            \quad\text{and}
            \\
            \label{eq:local:convergence:balance0:gammab2}
            \gamma_B
            &
            \ge
            \frac{\theta_{k+1}}{\theta_k}
            + \frac{\eta_k\SupNorm(\opt u)}{2\epsilon_u\theta_k}
            + \frac{\tilde C_x^2 \delta_{uw}^2 \eta_k}{4 \epsilon_w \mu \theta_k}.
        \end{align}
    \end{subequations}
    Then \eqref{eq:ineq-Hk-norm-1} holds.
\end{lemma}

\begin{proof}
    We follow the proof of \cref{lem:ineq-Hk-norm-1} until the estimate \eqref{eq:ineq-Hk-norm-1:final:2}, which now holds with $C_x=\tilde C_x\norm{x-\optx}_X$ and any  $\tilde\epsilon_u,\tilde\epsilon_w,\tilde\mu>0$ standing for $\epsilon_u,\epsilon_w,\mu>0$.
    Recall that we abbreviate $u=\nextu$, $w=\nextw$, and $x=\nextx$.
    Using Young's inequality and \eqref{eq:local:ineq-Hk-norm-1:apriori-bound}, we continue from there estimating that
    \begin{equation}
        \nonumber
        \begin{aligned}[t]
        \eta_k&\iprod{\xop(u, w)-\xop(\optu, \optw)}{x-\opt x}
        \\
        &
        \ge
        - \eta_k\left(\frac{\SupNorm(\optu)}{4\tilde\epsilon_u}+\frac{\tilde C_x \norm{x-\optx} \tilde\mu }{2}\right)\norm{w-\optw}_W^2
        - \eta_k\left(\frac{\SupNorm(\optw)}{4\tilde\epsilon_w}+\frac{\tilde C_x \norm{x-\optx}}{2 \tilde\mu }\right)\norm{u - \opt u}_U^2
        \\
        \MoveEqLeft[-1]
        - \eta_k(\tilde\epsilon_u+\tilde\epsilon_w)\norm{x - \opt x}_X^2
        \\
        &
        \ge
        - \eta_k\left(
            \frac{\SupNorm(\optu)}{4\tilde\epsilon_u}
            + \frac{\tilde C_x^2 \delta_{uw}^2 \tilde\mu^2 }{8 \tilde\epsilon_u}
        \right)
        \norm{w-\optw}_W^2
        - \eta_k\left(
            \frac{\SupNorm(\optw)}{4\tilde\epsilon_w}
            + \frac{\tilde C_x^2 \delta_{uw}^2 }{8 \tilde\mu^2 \tilde\epsilon_w }
        \right)
        \norm{u - \opt u}_U^2
        \\
        \MoveEqLeft[-1]
        - \eta_k(2\tilde\epsilon_u+2\tilde\epsilon_w)\norm{x - \opt x}_X^2.
        \end{aligned}
    \end{equation}
   With $\epsilon_u = 2\tilde\epsilon_u$, $\epsilon_u=2\tilde\epsilon_w$, and $\mu=\tilde\mu^2$, we now continue with the proof of \cref{lemma:convergence:main-estimate}, which goes through with \eqref{eq:local:convergence:balance0} in place of \eqref{eq:convergence:balance0}.
\end{proof}

\begin{lemma}[Local version of \cref{lemma:convergence:simplify-assumption}]
    \label{lemma:local:convergence:simplify-assumption}
    Suppose $\gamma_F > \tilde\gamma_F > 0$ as well as $\gamma_{G^*} \ge \tilde\gamma_{G^*} \ge 0$ and that there exist $\omega, t > 0 $ with $\omega\eta_{k+1} \le \eta_k$ for all $k \in \N$ such that
    \begin{equation}
        \label{eq:local:convergence:balance}
        \gamma_B
        \ge
        \inv\omega
        + tC_Q
        + \frac{4(1 + \inv t)}{\omega(\gamma_F-\tilde\gamma_F)^2}
        \left(
            \SupNorm(\optu)\pi_w
            + t\SupNorm(\optw)\pi_u
            + \frac{1}{4}\sqrt{t\pi_w\pi_u}\tilde C_x^2(\gamma_F-\tilde\gamma_F)\delta_{uw}^2
        \right).
    \end{equation}%
    Then there exist $\epsilon_u,\epsilon_w,\mu>0$, and, for all $k \in \N$, $\lambda_k,\theta_k>0$, such that \eqref{eq:local:convergence:balance0} holds.
\end{lemma}

\begin{proof}
    In the proof of \cref{lemma:convergence:simplify-assumption}, we replace $C_x$ by $\tilde C_x^2 \delta_{uw}^2$, and use \eqref{eq:local:convergence:balance0} in place of \eqref{eq:convergence:balance0} and \eqref{eq:local:convergence:balance} in place of \eqref{eq:convergence:balance}. Observe that compared to \eqref{eq:convergence:balance0:gammab} and \eqref{eq:convergence:balance0:gammab2}, \eqref{eq:local:convergence:balance0:gammab} and  \eqref{eq:local:convergence:balance0:gammab2} have an additional factor $2$ in front of the terms involving $\epsilon_u$ and $\epsilon_w$. This difference produces  the constant factors 4 instead of 2 in \eqref{eq:local:convergence:balance} compared to \eqref{eq:convergence:balance}.
\end{proof}

\begin{lemma}[Local version of \cref{lemma:convergence:main-estimate}]
    \label{lemma:local:convergence:main-estimate}
    Suppose \cref{assump:local:convergence:structural,assump:convergence:testing} hold as do \cref{assump:convergence:splitting} and \eqref{eq:local:convergence:balance0} for $k=0,\ldots,N-1$ with
    \begin{equation}
        \label{eq:local:convergence:deltuw}
        \delta_{uw}^2
        =
        \frac{1}{\gamma_B}
        \max\left\{
            \frac{1}{\lambda_0},
            \frac{C_Q\inv\gamma_B}{\lambda_0},
            \frac{1}{\theta_0},
            \frac{1+C_Q\inv\gamma_B}{\lambda_0+\theta_0}
        \right\}
        \RealXStep^2
    \end{equation}
    and
    \begin{equation}
        \label{eq:local:convergence:realxstep}
        \RealXStep \defeq \norm{v^0 - \opt v}_{Z_0\tilde M_0}.
    \end{equation}
    Also suppose $\{\lambda_k\}_{k \in \N}$ and $\{\theta_k\}_{k \in \N}$ are non-decreasing.
    Given $v^0$, let $v^1,\ldots,v^{N-1}$ be produced by \cref{alg:alg}.
    Then  \eqref{eq:convergence:quantitative-fejer} holds for $k=0,\ldots,N-1$, where all the terms are non-negative.
\end{lemma}

\begin{proof}
    We need to prove \eqref{eq:local:ineq-Hk-norm-1:apriori-bound} for all $k=0,\ldots,N-1$.
    The rest follows as in the proof of \cref{lemma:convergence:main-estimate}.

    \Cref{assump:convergence:splitting}\,\cref{item:convergence:splitting:principal} with \eqref{eq:convergence:zktildemk} and \cref{lemma:convergence:zm-lower-bound} establish for all $k=0,\ldots,N-1$ the \emph{a priori} bounds
    \begin{gather}
        \label{eq:local:apriori:u}
        \begin{aligned}[t]
        \norm{\nextu-\optu}_U^2
        &
        \le \frac{1}{\gamma_B} \left(\norm{\thisu-\optu}_U^2 + \pi_u\norm{\thisx-\optx}_X^2\right)
        \\
        &
        \le
        \frac{1}{\gamma_B}
        \max\left\{
            \frac{1}{\lambda_k},
            \frac{\pi_u}{\phi_k(1-\kappa)+(\lambda_k+\theta_k)\pi_u}
        \right\}
        \norm{\thisv-\opt v}_{Z_k\tilde M_k}^2
        \\
        &
        \le
        \frac{1}{\gamma_B}
        \max\left\{
            \frac{1}{\lambda_0},
            \frac{1}{\lambda_0+\theta_0}
        \right\}
        \norm{\thisv-\opt v}_{Z_k\tilde M_k}^2
        \\
        &
        \le
        \frac{\delta_{uw}^2}{\RealXStep^2}\norm{\thisv-\opt v}_{Z_k\tilde M_k}^2
        \end{aligned}
    \shortintertext{and}
        \label{eq:local:apriori:w}
        \begin{aligned}[t]
        \norm{\nextw-\optw}_W^2
        &
        \le \frac{1}{\gamma_B}\left(\norm{\thisw-\optw}_W^2 + C_Q \norm{\nextu-\optu}_U^2 + \pi_w\norm{\thisx-\optx}_X^2\right)
        \\
        &
        \le \frac{1}{\gamma_B}\left(
            \norm{\thisw-\optw}_W^2 + C_Q\inv\gamma_B\norm{\thisu-\optu}_U^2 +  (1 + C_Q \inv\gamma_B)\pi_w\norm{\thisx-\optx}_X^2
        \right)
        \\
        &
        \le
        \frac{1}{\gamma_B}
        \max\left\{
            \frac{1}{\theta_k},
            \frac{C_Q\inv\gamma_B}{\lambda_k},
            \frac{(1+C_Q\inv\gamma_B)\pi_w}{\phi_k(1-\kappa)+(\lambda_k+\theta_k)\pi_w}
        \right\}
        \norm{\thisv-\opt v}_{Z_k\tilde M_k}^2
        \\
        &
        \le
        \frac{1}{\gamma_B}
        \max\left\{
            \frac{1}{\theta_0},
            \frac{C_Q\inv\gamma_B}{\lambda_0},
            \frac{1+C_Q\inv\gamma_B)}{\lambda_0+\theta_0}
        \right\}
        \norm{\thisv-\opt v}_{Z_k\tilde M_k}^2
        \\
        &
        \le
        \frac{\delta_{uw}^2}{\RealXStep^2}\norm{\thisv-\opt v}_{Z_k\tilde M_k}^2.
        \end{aligned}
    \end{gather}
    In the final steps we have used the the assumptions that $\{\phi_k\}_{k \in \N}$ (by \cref{assump:convergence:testing}), $\{\lambda_k\}_{k \in \N}$, and $\{\theta_k\}_{k \in \N}$ are non-decreasing.

    We now use induction.
    By definition we have $\norm{v^0 - \opt v}_{Z_0\tilde M_0} \le \delta$.
    Hence \eqref{eq:local:apriori:u} and \eqref{eq:local:apriori:w}  verify \eqref{eq:local:ineq-Hk-norm-1:apriori-bound} for $k=0$.
    Suppose then that we have proved \eqref{eq:local:ineq-Hk-norm-1:apriori-bound} for $k=0,\ldots,\ell-1$. Then \eqref{eq:convergence:quantitative-fejer} holds $k=0,\ldots,\ell-1$ by following the proof of \cref{lemma:convergence:main-estimate}, replacing \cref{lem:ineq-Hk-norm-1} there in by the localized \cref{lem:local:ineq-Hk-norm-1}.
    Summing \eqref{eq:convergence:quantitative-fejer} over $k=0,\ldots,\ell-1$, we now obtain the \emph{a posteriori} bound
    \[
        \frac12
        \norm{v^\ell-\opt v}_{Z_{\ell}\tilde M_{\ell}}^2
        \leq
        \frac12\norm{v^0 - \opt v}_{Z_0\tilde M_0}^2
        =
        \frac{1}{2}\RealXStep^2.
    \]
    Now  \eqref{eq:local:apriori:u} and \eqref{eq:local:apriori:w} verify \eqref{eq:local:ineq-Hk-norm-1:apriori-bound} for $k=\ell$.
    Hence also \eqref{eq:convergence:quantitative-fejer} holds for $k=\ell$.
    As a result of the entire inductive argument, it holds for all $k=0‚\ldots,N-1$.
\end{proof}

With $\phi_0=1$ and the choices of
\[
    \lambda_0 \defeq \inv t r_0 \inv\pi_u \eta_0
    \quad\text{and}\quad
    \theta_0 \defeq r_0 \inv\pi_w \eta_0
    \quad\text{for}\quad
    r_0 \defeq \frac{\gamma_F-\tilde\gamma_F}{2(\inv t + 1)c_0}
\]
and $c_0 \defeq \eta_1/\eta_0 \le \inv\omega$ in the proof of \cref{lemma:local:convergence:simplify-assumption} (\cref{lemma:convergence:simplify-assumption}), we expand and estimate \eqref{eq:local:convergence:deltuw} as
\begin{equation}
    \label{eq:local:convergence:deltauw:expanded}
    \begin{aligned}[t]
    \delta_{uw}^2
    &
    =
    \frac{1}{\gamma_B}
    \max\left\{
        \frac{t\pi_u}{r_0\eta_0},
        \frac{tC_Q\pi_u}{r_0\eta_0\gamma_B},
        \frac{\pi_w}{r_0\eta_0},
        \frac{1+C_Q\inv\gamma_B}{(\inv t\inv\pi_u + \inv\pi_w)r_0\eta_0}
    \right\}
    \RealXStep^2
    \\
    &
    =
    \frac{t}{\gamma_Br_0\eta_0}
    \max\left\{
        \pi_u,
        \frac{C_Q\pi_u}{\gamma_B},
        \frac{\pi_w}{t},
        (1+C_Q\inv\gamma_B)\frac{\pi_u\pi_w}{\pi_w + t\pi_u}
    \right\}
    \RealXStep^2
    \\
    &
    \le
    \frac{2(1+t)}{\gamma_B\eta_0\omega(\gamma_F - \tilde\gamma_F)}
    \max\left\{
        \pi_u,
        \frac{C_Q\pi_u}{\gamma_B},
        \frac{\pi_w}{t},
        (1+C_Q\inv\gamma_B)\pi_u
    \right\}
    \RealXStep^2.
    \end{aligned}
\end{equation}
Hence \eqref{eq:local:convergence:balance} with $\delta_{uw}^2$ replaced by this upper estimate and $\phi_0=1$ (so that $\eta_0=\tau_0$) reads
\begin{subequations}
\label{eq:local:convergence:balance:estimated}
\begin{gather}
    \gamma_B
    \ge
    \inv\omega
    + tC_Q
    + \frac{4(1 + \inv t)}{\omega(\gamma_F-\tilde\gamma_F)^2}
    \left(\textstyle
        \SupNorm(\optu)\pi_w
        + t\SupNorm(\optw)\pi_u
        + \frac{\max\left\{
            \pi_u,
            \frac{C_Q\pi_u}{\gamma_B},
            \frac{\pi_w}{t},
            (1+C_Q\inv\gamma_B)\pi_u
        \right\}(1+t)\sqrt{t\pi_w\pi_u}}{2\gamma_B\tau_0\omega}
        \tilde C_x^2\RealXStep^2
    \right).
\shortintertext{where we recall that $t>0$ is a free balancing parameter, and}
    \RealXStep \defeq \norm{v^0 - \opt v}_{Z_0\tilde M_0}.
\end{gather}
\end{subequations}

We now immediately obtain local versions of the main results. By initializing close enough to a solution, i.e., with small $\delta$, we can possibly obtain convergence more often than from the global versions.

\begin{corollary}[Local accelerated convergence]
    In \cref{thm:convergence:accel}, replace \cref{assump:convergence:structural} by \cref{assump:local:convergence:structural} and \eqref{eq:convergence:balance-accel:2} by \eqref{eq:local:convergence:balance:estimated} with $\omega=\omega_0$.
    Then the claims continue to hold.
\end{corollary}

\begin{corollary}[Local linear convergence]
    In \cref{thm:convergence:linear}, replace \cref{assump:convergence:structural} by \cref{assump:local:convergence:structural} and \eqref{eq:convergence:balance-linear:2} and \eqref{eq:local:convergence:balance:estimated} with $\tau_0=\tau$.
    Then the claims continue to hold.
\end{corollary}

Both proofs are exactly as the original proofs, using \cref{lemma:local:convergence:main-estimate} in place of \cref{lemma:convergence:main-estimate}.


\begin{thebibliography}{10}

\bibitem{ambrosio2000fbv}
L{.\nobreak\kern 0.33333em}Ambrosio, N{.\nobreak\kern 0.33333em}Fusco, and
  D{.\nobreak\kern 0.33333em}Pallara, \emph{Functions of Bounded Variation and
  Free Discontinuity Problems}, Oxford University Press, 2000.

\bibitem{bosse2014oneshot}
T{.\nobreak\kern 0.33333em}Bosse, N.\,R{.\nobreak\kern 0.33333em}Gauger,
  A{.\nobreak\kern 0.33333em}Griewank, S{.\nobreak\kern 0.33333em}Günther, and
  V{.\nobreak\kern 0.33333em}Schulz, One-Shot Approaches to Design
  Optimization, \emph{Trends in PDE Constrained Optimization}  (2014),  43--66,
  \href{https://dx.doi.org/10.1007/978-3-319-05083-6_5}{\nolinkurl{doi:10.1007/978-3-319-05083-6_5}}.

\bibitem{brezis2011functional}
H{.\nobreak\kern 0.33333em}Brezis, \emph{Functional Analysis, Sobolev Spaces
  and Partial Differential Equations}, Springer, 2011,
  \href{https://dx.doi.org/10.1007/978-0-387-70914-7}{\nolinkurl{doi:10.1007/978-0-387-70914-7}}.

\bibitem{ckp1999regularization}
E{.\nobreak\kern 0.33333em}Casas, K{.\nobreak\kern 0.33333em}Kunisch, and
  C{.\nobreak\kern 0.33333em}Pola, Regularization by Functions of Bounded
  Variation and Applications to Image Enhancement, \emph{Applied Mathematics
  and Optimization} 40 (1999),  229--257,
  \href{https://dx.doi.org/10.1007/s002459900124}{\nolinkurl{doi:10.1007/s002459900124}}.

\bibitem{chambolle2011first}
A{.\nobreak\kern 0.33333em}Chambolle and T{.\nobreak\kern 0.33333em}Pock, A
  first-order primal-dual algorithm for convex problems with applications to
  imaging, \emph{Journal of Mathematical Imaging and Vision} 40 (2011),
  120--145,
  \href{https://dx.doi.org/10.1007/s10851-010-0251-1}{\nolinkurl{doi:10.1007/s10851-010-0251-1}}.

\bibitem{tuomov-nlpdhgm-redo}
C{.\nobreak\kern 0.33333em}Clason, S{.\nobreak\kern 0.33333em}Mazurenko, and
  T{.\nobreak\kern 0.33333em}Valkonen, Acceleration and global convergence of a
  first-order primal-dual method for nonconvex problems, \emph{SIAM Journal on
  Optimization} 29 (2019),  933--963,
  \href{https://dx.doi.org/10.1137/18M1170194}{\nolinkurl{doi:10.1137/18m1170194}},
  \href{https://arxiv.org/abs/1802.03347}{\nolinkurl{arXiv:1802.03347}}.

\bibitem{tuomov-nlpdhgm-general}
C{.\nobreak\kern 0.33333em}Clason, S{.\nobreak\kern 0.33333em}Mazurenko, and
  T{.\nobreak\kern 0.33333em}Valkonen, Primal-dual proximal splitting and
  generalized conjugation in nonsmooth nonconvex optimization, \emph{Applied
  Mathematics and Optimization}  (2020),
  \href{https://dx.doi.org/10.1007/s00245-020-09676-1}{\nolinkurl{doi:10.1007/s00245-020-09676-1}},
  \href{https://arxiv.org/abs/1901.02746}{\nolinkurl{arXiv:1901.02746}}.

\bibitem{tuomov-pdex2nlpdhgm}
C{.\nobreak\kern 0.33333em}Clason and T{.\nobreak\kern 0.33333em}Valkonen,
  Primal-dual extragradient methods for nonlinear nonsmooth {PDE}-constrained
  optimization, \emph{SIAM Journal on Optimization} 27 (2017),  1313--1339,
  \href{https://dx.doi.org/10.1137/16M1080859}{\nolinkurl{doi:10.1137/16m1080859}},
  \href{https://arxiv.org/abs/1606.06219}{\nolinkurl{arXiv:1606.06219}}.

\bibitem{clasonvalkonen2020nonsmooth}
C{.\nobreak\kern 0.33333em}Clason and T{.\nobreak\kern 0.33333em}Valkonen,
  Introduction to Nonsmooth Analysis and Optimization, 2020,
  \href{https://arxiv.org/abs/2001.00216}{\nolinkurl{arXiv:2001.00216}}.
\newblock Work in progress.

\bibitem{darde2021electrodeless}
J{.\nobreak\kern 0.33333em}Dardé, N{.\nobreak\kern 0.33333em}Hyvönen,
  T{.\nobreak\kern 0.33333em}Kuutela, and T{.\nobreak\kern 0.33333em}Valkonen,
  Contact adapting electrode model for electrical impedance tomography,
  \emph{SIAM Journal on Applied Mathematics} 82 (2022),  427--449,
  \href{https://dx.doi.org/10.1137/21M1396125}{\nolinkurl{doi:10.1137/21m1396125}},
  \href{https://arxiv.org/abs/2102.01926}{\nolinkurl{arXiv:2102.01926}}.

\bibitem{evans1998pde}
L.\,C{.\nobreak\kern 0.33333em}Evans, \emph{Partial Differential Equations},
  Americal Mathematical Society, 1998.

\bibitem{golub1996matrix}
G{.\nobreak\kern 0.33333em}Golub and C{.\nobreak\kern 0.33333em}Van~Loan,
  \emph{Matrix Computations}, Johns Hopkins Studies in the Mathematical
  Sciences, Johns Hopkins University Press, 1996.

\bibitem{griewank2006projected}
A{.\nobreak\kern 0.33333em}Griewank, Projected Hessians for Preconditioning in
  One-Step One-Shot Design Optimization, \emph{Large-Scale Nonlinear
  Optimization}  (2006),  151--–171,
  \href{https://dx.doi.org/10.1007/0-387-30065-1_10}{\nolinkurl{doi:10.1007/0-387-30065-1_10}}.

\bibitem{guenther2016simultaneous}
S{.\nobreak\kern 0.33333em}Günther, N.\,R{.\nobreak\kern 0.33333em}Gauger, and
  Q{.\nobreak\kern 0.33333em}Wang, Simultaneous single-step one-shot
  optimization with unsteady {PDE}s, \emph{Journal of Computational and Applied
  Mathematics} 294 (2016),  12--–22,
  \href{https://dx.doi.org/10.1016/j.cam.2015.07.033}{\nolinkurl{doi:10.1016/j.cam.2015.07.033}}.

\bibitem{hamdi2009reduced}
A{.\nobreak\kern 0.33333em}Hamdi and A{.\nobreak\kern 0.33333em}Griewank,
  Reduced quasi-{N}ewton method for simultaneous design and optimization,
  \emph{Computational Optimization and Applications} 49 (2009),  521--–548,
  \href{https://dx.doi.org/10.1007/s10589-009-9306-x}{\nolinkurl{doi:10.1007/s10589-009-9306-x}}.

\bibitem{hamdi2010properties}
A{.\nobreak\kern 0.33333em}Hamdi and A{.\nobreak\kern 0.33333em}Griewank,
  Properties of an augmented {L}agrangian for design optimization,
  \emph{Optimization Methods and Software} 25 (2010),  645--–664,
  \href{https://dx.doi.org/10.1080/10556780903270910}{\nolinkurl{doi:10.1080/10556780903270910}}.

\bibitem{hazra2004simultaneous}
S.\,B{.\nobreak\kern 0.33333em}Hazra and V{.\nobreak\kern 0.33333em}Schulz,
  Simultaneous Pseudo-Timestepping for {PDE}-Model Based Optimization Problems,
  \emph{BIT Numerical Mathematics} 44 (2004),  457--–472,
  \href{https://dx.doi.org/10.1023/b:bitn.0000046815.96929.b8}{\nolinkurl{doi:10.1023/b:bitn.0000046815.96929.b8}}.

\bibitem{he2012convergence}
B{.\nobreak\kern 0.33333em}He and X{.\nobreak\kern 0.33333em}Yuan, Convergence
  Analysis of Primal-Dual Algorithms for a Saddle-Point Problem: From
  Contraction Perspective, \emph{SIAM Journal on Imaging Sciences} 5 (2012),
  119--149,
  \href{https://dx.doi.org/10.1137/100814494}{\nolinkurl{doi:10.1137/100814494}}.

\bibitem{hintermuller2002primaldual}
M{.\nobreak\kern 0.33333em}Hinterm{\"u}ller, K{.\nobreak\kern 0.33333em}Ito,
  and K{.\nobreak\kern 0.33333em}Kunisch, The primal-dual active set strategy
  as a semismooth {N}ewton method, \emph{SIAM Journal on Optimization} 13
  (2002),  865--888 (2003),
  \href{https://dx.doi.org/10.1137/s1052623401383558}{\nolinkurl{doi:10.1137/s1052623401383558}}.

\bibitem{hintermuller2006infeasible}
M{.\nobreak\kern 0.33333em}Hinterm{\"u}ller and G{.\nobreak\kern
  0.33333em}Stadler, An Infeasible Primal-Dual Algorithm for Total Bounded
  Variation--Based Inf-Convolution-Type Image Restoration, \emph{SIAM Journal
  on Scientific Computation} 28 (2006),  1--23.

\bibitem{kunisch2008lagrange}
K{.\nobreak\kern 0.33333em}Ito and K{.\nobreak\kern 0.33333em}Kunisch,
  \emph{Lagrange {M}ultiplier {A}pproach to {V}ariational {P}roblems and
  {A}pplications}, volume~15 of Advances in Design and Control, SIAM, 2008,
  \href{https://dx.doi.org/10.1137/1.9780898718614}{\nolinkurl{doi:10.1137/1.9780898718614}}.

\bibitem{jauhiainen2019gaussnewton}
J{.\nobreak\kern 0.33333em}Jauhiainen, P{.\nobreak\kern 0.33333em}Kuusela,
  A{.\nobreak\kern 0.33333em}Seppänen, and T{.\nobreak\kern
  0.33333em}Valkonen, Relaxed {G}auss--{N}ewton methods with applications to
  electrical impedance tomography, \emph{SIAM Journal on Imaging Sciences} 13
  (2020),  1415--1445,
  \href{https://dx.doi.org/10.1137/20M1321711}{\nolinkurl{doi:10.1137/20m1321711}},
  \href{https://arxiv.org/abs/2002.08044}{\nolinkurl{arXiv:2002.08044}}.

\bibitem{jensen2022codes}
B{.\nobreak\kern 0.33333em}Jensen, Codes for ``A nonsmooth primal-dual method
  with interwoven PDE constraint solver'', 2022,
  \href{https://dx.doi.org/10.5281/zenodo.7398160}{\nolinkurl{doi:10.5281/zenodo.7398160}}.

\bibitem{kaland2013oneshot}
L{.\nobreak\kern 0.33333em}Kaland, J.\,C{.\nobreak\kern 0.33333em}De~Los~Reyes,
  and N.\,R{.\nobreak\kern 0.33333em}Gauger, One-shot methods in function space
  for {PDE}-constrained optimal control problems, \emph{Optimization Methods
  and Software} 29 (2013),  376--405,
  \href{https://dx.doi.org/10.1080/10556788.2013.774397}{\nolinkurl{doi:10.1080/10556788.2013.774397}}.

\bibitem{kreyszig1991introductory}
E{.\nobreak\kern 0.33333em}Kreyszig, \emph{Introductory Functional Analysis
  with Applications}, Wiley Classics Library, Wiley, 1991.

\bibitem{leveque2007fdm}
R{.\nobreak\kern 0.33333em}LeVeque, J., \emph{Finite Difference Methods for
  Ordinary and Partial Differential Equations}, SIAM, 2007,
  \href{https://dx.doi.org/10.1137/1.9780898717839}{\nolinkurl{doi:10.1137/1.9780898717839}}.

\bibitem{tuomov-nlpdhgm-block}
S{.\nobreak\kern 0.33333em}Mazurenko, J{.\nobreak\kern 0.33333em}Jauhiainen,
  and T{.\nobreak\kern 0.33333em}Valkonen, Primal-dual block-proximal splitting
  for a class of non-convex problems, \emph{Electronic Transactions on
  Numerical Analysis} 52 (2020),  509--552,
  \href{https://dx.doi.org/10.1553/etna_vol52s509}{\nolinkurl{doi:10.1553/etna_vol52s509}},
  \href{https://arxiv.org/abs/1911.06284}{\nolinkurl{arXiv:1911.06284}}.

\bibitem{mifflin1977semismooth}
R{.\nobreak\kern 0.33333em}Mifflin, Semismooth and semiconvex functions in
  constrained optimization, \emph{SIAM Journal on Control And Optimization} 15
  (1977),  959--972,
  \href{https://dx.doi.org/10.1137/0315061}{\nolinkurl{doi:10.1137/0315061}}.

\bibitem{mordukhovich2006variational}
B.\,S{.\nobreak\kern 0.33333em}Mordukhovich, \emph{Variational Analysis and
  Generalized Differentiation {I}: Basic Theory}, volume 330 of Grundlehren der
  mathematischen Wissenschaften, Springer, 2006,
  \href{https://dx.doi.org/10.1007/3-540-31247-1}{\nolinkurl{doi:10.1007/3-540-31247-1}}.

\bibitem{qi1993nonsmooth}
L{.\nobreak\kern 0.33333em}Qi and J{.\nobreak\kern 0.33333em}Sun, A nonsmooth
  version of {N}ewton's method, \emph{Mathematical Programming} 58 (1993),
  353--367,
  \href{https://dx.doi.org/10.1007/BF01581275}{\nolinkurl{doi:10.1007/bf01581275}}.

\bibitem{sirignano2022online}
J{.\nobreak\kern 0.33333em}Sirignano and K{.\nobreak\kern
  0.33333em}Spiliopoulos, Online Adjoint Methods for Optimization of {PDE}s,
  \emph{Applied Mathematics and Optimization} 85 (2022),
  \href{https://dx.doi.org/10.1007/s00245-022-09852-5}{\nolinkurl{doi:10.1007/s00245-022-09852-5}}.

\bibitem{suonpera2022bilevel}
E{.\nobreak\kern 0.33333em}Suonper\"a and T{.\nobreak\kern 0.33333em}Valkonen,
  Linearly convergent bilevel optimization with single-step inner methods,
  \emph{Computational Optimization and Applications}  (2023),
  \href{https://arxiv.org/abs/2205.04862}{\nolinkurl{arXiv:2205.04862}}.
\newblock accepted.

\bibitem{ta1991one}
S{.\nobreak\kern 0.33333em}Ta'asan, One Shot Methods for Optimal Control of
  Distributed Parameter Systems {I}: Finite Dimensional Control, Technical
  Report 91-2, Institute for Computer Applications in Science and Engineering,
  NASA Langley Research Center, 1991.

\bibitem{ulbrich2002semismooth}
M{.\nobreak\kern 0.33333em}Ulbrich, Semismooth {N}ewton methods for operator
  equations in function spaces, \emph{SIAM Journal on Optimization} 13 (2002),
  805--842 (2003),
  \href{https://dx.doi.org/10.1137/s1052623400371569}{\nolinkurl{doi:10.1137/s1052623400371569}}.

\bibitem{ulbrich2011semismooth}
M{.\nobreak\kern 0.33333em}Ulbrich, \emph{Semismooth {N}ewton Methods for
  Variational Inequalities and Constrained Optimization Problems in Function
  Spaces}, volume~11 of MOS-SIAM Series on Optimization, SIAM, 2011,
  \href{https://dx.doi.org/10.1137/1.9781611970692}{\nolinkurl{doi:10.1137/1.9781611970692}}.

\bibitem{tuomov-nlpdhgm}
T{.\nobreak\kern 0.33333em}Valkonen, A primal-dual hybrid gradient method for
  non-linear operators with applications to {MRI}, \emph{Inverse Problems} 30
  (2014),  055012,
  \href{https://dx.doi.org/10.1088/0266-5611/30/5/055012}{\nolinkurl{doi:10.1088/0266-5611/30/5/055012}},
  \href{https://arxiv.org/abs/1309.5032}{\nolinkurl{arXiv:1309.5032}}.

\bibitem{tuomov-proxtest}
T{.\nobreak\kern 0.33333em}Valkonen, Testing and non-linear preconditioning of
  the proximal point method, \emph{Applied Mathematics and Optimization} 82
  (2020),
  \href{https://dx.doi.org/10.1007/s00245-018-9541-6}{\nolinkurl{doi:10.1007/s00245-018-9541-6}},
  \href{https://arxiv.org/abs/1703.05705}{\nolinkurl{arXiv:1703.05705}}.

\bibitem{tuomov-regtheory}
T{.\nobreak\kern 0.33333em}Valkonen, Regularisation, optimisation,
  subregularity, \emph{Inverse Problems} 37 (2021),  045010,
  \href{https://dx.doi.org/10.1088/1361-6420/abe4aa}{\nolinkurl{doi:10.1088/1361-6420/abe4aa}},
  \href{https://arxiv.org/abs/2011.07575}{\nolinkurl{arXiv:2011.07575}}.

\bibitem{vilhunen2002simultaneous}
T{.\nobreak\kern 0.33333em}Vilhunen, J.\,P{.\nobreak\kern 0.33333em}Kaipio,
  P.\,J{.\nobreak\kern 0.33333em}Vauhkonen, T{.\nobreak\kern
  0.33333em}Savolainen, and M{.\nobreak\kern 0.33333em}Vauhkonen, Simultaneous
  reconstruction of electrode contact impedances and internal electrical
  properties: {I}. {T}heory, \emph{Meas. Sci. Technol.} 13 (2002),  1848--1854.

\end{thebibliography}
\end{document}